\documentclass[10pt]{amsart}

\usepackage[all]{xy}
\usepackage{xspace}
\usepackage{graphicx}
\usepackage{longtable}
\usepackage{mathtools}
\usepackage{latexsym}
\usepackage{stmaryrd}
\usepackage{fullpage}
\usepackage{amsfonts,amsmath,amscd,amssymb}
\input{xy}
\xyoption{all}
\xyoption{curve}

\newcommand{\bC}{{\mathbb C}}

\newcommand{\bG}{{\mathbb G}}
\newcommand{\bK}{{\mathbb K}}
\newcommand{\bF}{{\mathbb F}}
\newcommand{\bQ}{{\mathbb Q}}
\newcommand{\bZ}{{\mathbb Z}}
\newcommand{\bN}{{\mathbb N}}
\newcommand{\bU}{{\mathbb U}}

\newcommand{\fg}{{\mathfrak g}}
\newcommand{\fl}{{\mathfrak l}}
\newcommand{\fh}{{\mathfrak h}}

\newcommand{\fn}{{\mathfrak n}}

\newcommand{\fs}{{\mathfrak s}}

\newcommand{\fb}{{\mathfrak b}}

\newcommand{\gr}{{{\mbox{\rm gr}}}}

\newcommand{\ad}{{{\mbox{\rm ad}}}}
\newcommand{\Ad}{{{\mbox{\rm Ad}}}}
\newcommand{\Char}{{{\mbox{\rm char}}}}
\newcommand{\Lie}{{{\mbox{\rm Lie}}}} 
 
\newcommand{\Hom}{{{\mbox{\rm Hom}}}} 
\newcommand{\Irr}{{{\mbox{\rm Irr}}}}
\newcommand{\Id}{{{\mbox{\rm Id}}}}

\newcommand{\Img}{{{\mbox{\rm Im}}}}

\newcommand{\hcf}{{{\mbox{\rm hcf}}}}

\newcommand{\End}{{{\mbox{\rm End}}}}
\newcommand{\Di}{{{\mbox{\rm Dist}}}}

\renewcommand{\2}{_{(2)}}

\newcommand{\ve}{\ensuremath{\mathbf{e}}\xspace}

\newcommand{\vf}{\ensuremath{\mathbf{f}}\xspace}

\newcommand{\vh}{\ensuremath{\mathbf{h}}\xspace}

\newcommand{\twopartdef}[4]{\left\{
	\begin{array}{ll}
		#1 & \mbox{if } #2 \\
		#3 & \mbox{if } #4
	\end{array}
	\right.}

\newcommand{\threepartdef}[6]{\left\{
	\begin{array}{ll}
		#1 & \mbox{if } #2 \\
		#3 & \mbox{if } #4 \\
		#5 & \mbox{if } #6
	\end{array}
	\right.}

\newcommand{\Upr}{U^{[r]}(G)}
\newcommand{\Uprc}{U_\chi^{[r]}(G)}
\newcommand{\Dipri}{\Di^{+}_{p^r}(G)}

\newcommand{\Diprj}{\Di^{+}_{p^{r+1}-1}(G)}

\newtheorem{theorem}{Theorem}[subsection]
\newtheorem{prop}[theorem]{Proposition}
\newtheorem{lemma}[theorem]{Lemma}
\newtheorem*{dfn}{Definition}
\newtheorem{cor}[theorem]{Corollary}

\newtheorem{exa}{Example}
\newtheorem*{rmk}{Remark}
\newtheorem*{conj}{Conjecture}

\begin{document}
\title[Higher Deformations]{Higher Deformations of Lie Algebra Representations I}
\author{Matthew Westaway}
\email{M.P.Westaway@warwick.ac.uk}
\address{Department of Mathematics, University of Warwick, Coventry, CV4 7AL, UK}
\thanks{This paper has been accepted for publication in the Journal of the Mathematical Society of Japan.}
\date{January 7, 2019}
  
\begin{abstract}
	In the late 1980s, Friedlander and Parshall studied the representations of a family of algebras which were obtained as deformations of the distribution algebra of the first Frobenius kernel of an algebraic group. The representation theory of these algebras tells us much about the representation theory of Lie algebras in positive characteristic. We develop an analogue of this family of algebras for the distribution algebras of the higher Frobenius kernels, answering a 30 year old question posed by Friedlander and Parshall. We also examine their representation theory in the case of the special linear group.
\end{abstract}

\maketitle

\section{Introduction}
\label{sec1}
In 1988 and 1990, Eric Friedlander and Brian Parshall published a pair of papers \cite{PF2,PF1} which have had a great impact on our understanding of the representation theory of Lie algebras over algebraically-closed fields of positive characteristic. Prior to the publication of these papers, it was known that in characteristic $p>0$ many of the Lie algebras that were interesting to study came with a so-called $p$-structure: a map $\fg\to\fg$ which gave a notion of $p$-th powers to elements of $\fg$. When considering the $p$-structure-preserving representations of these Lie algebras, which are called restricted representations, it was discovered that these were in 1-1 correspondence with representations of the first Frobenius kernel $G_{(1)}$ of $G$, in the case when $\fg$ was the Lie algebra of an algebraic group $G$. 

Friedlander and Parshall, however, were interested in the general representation theory of $\fg$, rather than the restricted representation theory. Their method was to use an observation of Kac and Weisfeiler in \cite{KW2} that from the universal enveloping algebra $U(\fg)$ one could construct a family of algebras, which they denoted $A_\chi$ and are today written as $U_\chi(\fg)$, indexed by the linear forms on $\fg$.
Every irreducible (not necessarily restricted) $\fg$-module appears as a $U_\chi(\fg)$-module for some $\chi\in\fg^{*}$. In the case when $\chi=0$, one recovers precisely the restricted representation theory of $\fg$.

At the end of \cite{PF2}, Friedlander and Parshall pose a number of questions about these algebras and their representation theory. Question 5.4 in that paper, which was posed to them in turn by J. Humphreys, is as follows:

\textit{{\bf Hyperalgebra analogues.} Do the algebras $A_\chi$ have natural analogues corresponding to the infinitesimal group schemes $G_r$ associated to $G$ for $r>1$?}


It is this question which we answer here. To do such, we must first define and study a family of higher universal enveloping algebras $U^{[r]}(G)$ for $r\in\bN$, analogues of the universal enveloping algebra in these higher cases. When $r=0$, this algebra is precisely $U(\fg)$ (where $\fg=\Lie(G)$), and the family of algebras $\{\Upr\}_{r\in\bN}$ form a direct system with limit $\Di(G)$ (the distribution algebra of $G$ \cite{Jan}). This family of algebras was first introduced by Masaharu Kaneda and Jiachen Ye in \cite{KY}, however their study of it related primarily to its connection to the study of arithmetic differential operators \cite{Ber}. The sum and substance of their results on the structure of this algebra can be found in Subsection~\ref{sec2.2} of this paper, and this algebra has been minimally studied since then. Indeed, Kaneda and Ye's construction is not especially useful for the goals of this paper and we shall define the algebra $U^{[r]}(G)$ in a different way, before showing that these constructions are isomorphic in Subsection~\ref{sec3.3}.

The majority of the results in this paper are proved in the case when $G$ is a reductive algebraic group. This restriction is not unusual in this area of study -- indeed many of the most notable reviews of this subject make the same restriction fairly early on (see \cite{Hum,Jan2}). Nevertheless, this paper requires reductivity sooner than is typical, and in fact several of the results proved in this paper shall hold without this assumption.  This is proved in the sequel to this paper \cite{West}, where the Hopf algebra structure of the algebras $\Upr$ is studied in greater detail. This direction would appear to be the most fruitful in studying these algebras for more general algebraic groups.

When $G$ is reductive, the higher universal enveloping algebras $\Upr$ share many similarities with the universal enveloping algebras. They are finitely generated over their centre (Proposition~\ref{fingen}), all of their irreducible modules are finite-dimensional (Theorem~\ref{findim}), and they have a PBW basis (Proposition~\ref{basis}). In fact, there exist surjective Hopf algebra homomorphisms $\Upr\to U(\fg)$ for each $r\in\bN$ by Lemma~\ref{equiv} and Corollary~\ref{surj}. Furthermore, Lemma~\ref{powers} enables us to define a notion of $p$-th powers in these algebras, and hence to define the algebras $\Uprc$ indexed by $\chi\in\fg^{*}$. These $\Uprc$ are the analogues of the $U_\chi(\fg)$ in this higher setting, and every irreducible $\Upr$-module is an irreducible $\Uprc$-module for some $\chi\in \fg^{*}$ (Lemma~\ref{irred}).


In studying the representation theory of these $\Uprc$, one can define the notion of a \emph{higher baby Verma module} $Z_\chi^r(\lambda)$ analogously to the construction in the standard case. One obtains that every irreducible $\Uprc$-module is an irreducible quotient of a higher baby Verma module (Lemma~\ref{Verma}), however in comparison with the standard case these modules are often too large to pinpoint the irreducible modules explicitly. For example, when $G=SL_2$ and $\chi\neq 0$ the baby Verma modules are always irreducible -- this ceases to be true for the higher baby Verma modules. The irreducible modules for $U^{[r]}_\chi(SL_2)$ are characterised in Theorem~\ref{Class}, where we see that a different module construction, called \emph{teenage Verma modules}, behaves as the baby Verma modules do in the standard case. It is conjectured in Section~\ref{sec7} that these modules provide the correct analogue of baby Verma modules for the higher universal enveloping algebras -- these conjectures are proved in the sequel to this paper \cite{West}. 

One interesting feature of the higher universal enveloping algebras is that for $r\in\bN$ the finite-dimensional Hopf algebra $\Di(G_{(r)})$ is a normal Hopf subalgebra of $\Upr$ (here $G_{(r)}$ is the $r$-th Frobenius kernel of $G$ and is an infinitesimal group scheme). When $r=0$ this is automatic as $\Di(G_{(0)})=\bK$, but when $r>0$ this adds new complexity to these algebras. One application of this fact is Theorem~\ref{Decomp}, which shows that every irreducible $\Upr$-module $M$ contains a unique irreducible $\Di(G_{(r)})$-submodule $N$ and that $M\cong N\otimes V$ as $\Di(G_{(r)})$-modules for a finite-dimensional vector space $V$. This allows us to interpret Kac-Weisfeiler's second conjecture \cite{KW,Prem} in this context: if $M$ is a $\Uprc$-module for $\chi\in\fg^{*}$, does $p^{\dim (G\cdot\chi)/2}$ divide the dimension of $V$? We answer this in Theorem~\ref{prem} and in the sequel \cite{West}.

The structure of this paper is as follows. We start in Section~\ref{sec2} by recalling the various definitions of enveloping algebras for a Lie algebra over a field of characteristic $p>0$, as well as examining the different notions for differential operators in this context. Then, in Section~\ref{sec3}, we introduce the algebra $U^{[r]}(G)$ which we study for the rest of the paper. We develop the appropriate analogue of $p$-structure and $p$th powers in this context and construct a basis for the higher universal enveloping algebras. We restrict to reductive algebraic groups midway into this section. In Section~\ref{sec4} we show the connection between $U^{[r]}(G)$ and the standard universal enveloping algebra $U(\fg)$. We then move on to studying the representation theory of $U^{[r]}(G)$ in Section~\ref{sec6}, which allows us to define the family of algebras $U^{[r]}_\chi(G)$, as well as higher notions of baby Verma modules. In Section~\ref{sec7} we focus specifically on the case of $G=SL_2$ and try to understand the representation theory of the $U^{[r]}_\chi(SL_2)$; in particular seeing how it differs from the well-understood case $r=0$ as studied by Friedlander and Parshall \cite{PF2,PF1}. Finally, in Section~\ref{sec9} we give some results on the Hopf algebraic structure of the higher universal enveloping algebras.

I would like to thank my PhD supervisors Dmitriy Rumynin and Inna Capdeboscq for suggesting this topic to me and for their continued assistance with the production of this paper. I would also like to thank Lewis Topley for some valuable conversations regarding this paper. Furthermore, I would like to thank the referee for their valuable comments.

\section{Preliminaries}
\label{sec2}

\subsection{Universal Enveloping Algebras}
\label{sec2.1}

Let $G$ be an algebraic group over an algebraically closed field $\bK$ of characteristic $p>0$, and let $\fg=\Lie(G)$. In positive characteristic, there are several sensible notions for an enveloping algebra of $\fg$, all of which are isomorphic when the characteristic is zero. Let us briefly recall their constructions.

Firstly, we can construct the universal enveloping algebra
$$U(\fg)\coloneqq \frac{T(\fg)}{Q},$$
where $T(\fg)$ is the tensor algebra of $\fg$ and $Q$ is the 2-sided ideal generated by the elements $x\otimes y -y\otimes x - [x,y]$ for $x,y\in\fg$.

Since $\fg$ is constructed here as the Lie algebra of an algebraic group, it has a $p$-structure \cite{Jan2}. That is, there exists a map $\,^{[p]}:\fg\to\fg$ such that the map $\xi:\fg\to U(\fg)$ given by $x\mapsto x^p-x^{[p]}$ satisfies the following two conditions: (1) the image lies inside $Z(U(\fg))$, and (2) that $\xi(ax+by)=a^p\xi(x)+b^p\xi(y)$ for all $x,y\in\fg$, $a,b\in\bK$. This allows us to form the algebra
$$U_0(\fg)\coloneqq\frac{U(\fg)}{\langle x^p-x^{[p]}\,\vert\,x\in\fg\rangle},$$
called the restricted enveloping algebra of $\fg$. 

Let us now recall the definition of the distribution algebra $\Di(G)$. If $I_1\coloneqq \{f\in\bK[G]\,\vert\,f(1)=0\}$ (where $1$ is the identity of $G$), then we denote $\Di_k(G)\coloneqq\{\mu:\bK[G]\rightarrow \bK\,\vert\,\mu\,\mbox{is linear and}\,\mu(I_1^{k+1})=0\}$ and $\Di_k^{+}(G)\coloneqq\{\mu\in\Di_k(G)\,\vert\,\mu(1)=0\}$. We then denote $\Di(G)=\cup_{k\geq 0}\Di_k(G)$. 

$\Di(G)$ is an algebra, and $\fg$ lies inside $\Di(G)$ as $\Di_1^{+}(G)$. The Lie bracket on $\fg$ corresponds to the Lie bracket $[A,B]=AB-BA$ on $\Di(G)$. Recall that if $\mu\in\Di^{+}_i(G)$ and $\rho\in\Di^{+}_j(G)$ then $\mu\rho\in\Di^{+}_{i+j}(G)$ and $[\mu,\rho]\in\Di^{+}_{i+j-1}(G)$ \cite{Jan}.

The distribution algebra is related to the previous enveloping algebras by the following observation: $U_0(\fg)$ is isomorphic to $\Di(G_{(1)})$, where we denote by $G_{(1)}$ the first Frobenius kernel of $G$. Throughout the paper we shall more generally denote by $G_{(r)}$ the $r$th Frobenius kernel of $G$ (see \cite{Jan}). 

\subsection{Differential Operators}
\label{sec2.2}

When studying sheaves of differential operators on a smooth variety over an algebraically closed field of positive characteristic there are several distinct notions, which coincide in zero characteristic. Firstly, there are the differential operators constructed by Grothendieck in \cite{EGA}. The precise construction is omitted here, but the reader should consult \cite{EGA} for more detail. In particular, the sheaf $\mathcal{D}iff_{X/\bK}$ of these differential operators lies inside the sheaf $\mathcal{E}nd_{\bK}(\mathcal{O}_X)$.

This sheaf has a filtration
$$\mathcal{D}_{X/\bK}^{(0)}\to \mathcal{D}_{X/\bK}^{(1)} \to\ldots \to\mathcal{D}_{X/\bK}^{(m)}\to \ldots\to\mathcal{D}iff_{X/\bK}=\varinjlim\mathcal{D}_{X/\bK}^{(m)}$$ constructed by Berthelot in \cite{Ber}. This sheaf $\mathcal{D}_{X/\bK}^{(0)}$ is called the sheaf of {\em crystalline differential operators} and was constructed by Berthelot before the rest of the filtration was developed. The sheaf was used by Bezrukavnikov, Mirkovi\'{c} and Rumynin in \cite{BMR} where they use it to derive a version of Beilinson-Bernstein's localisation theorem in positive characteristic. The sheaves $\mathcal{D}_{X/\bK}^{(m)}$ are called the sheaves of {\em arithmetic differential operators}.

When $X=G$ is a smooth algebraic group we can compare the sheaves of differential operators with the above notions of universal enveloping algebras. In particular, there is an injective algebra homomorphism $\Di(G)\hookrightarrow \Gamma(G,\mathcal{D}iff_{G/\bK})$, which is an isomorphism onto the subalgebra of left invariant differential operators. See \cite[I.7.18]{Jan} for details. Similarly, there is an injective algebra homomorphism $U(\fg)\hookrightarrow \Gamma(G,\mathcal{D}_{X/\bK}^{(0)})$ which is an isomorphism onto the left invariant crystalline differential operators.

In trying to construct the analogues to the $U_\chi(\fg)$ from Friedlander and Parshall's question, one sees that the arithmetic differential operators should play a role. To work with arithmetic differential operators explicitly, it helps to recall from \cite{HJR} that
$$\mathcal{D}_{X/\bK}^{(m)}\cong \frac{T_{\bK}(\mathcal{D}iff^{2p^m-1})}{\langle \lambda-\lambda 1_{\mathcal{O}_X},\,\,\delta\otimes\delta' - \delta'\otimes \delta - [\delta,\delta'],\,\,\,\delta\otimes\delta''-\delta\delta''\,\,\vert\,\,\lambda\in\bK,\,\delta''\in\mathcal{D}iff^{p^m-1},\,\,\delta,\delta'\in\mathcal{D}iff^{p^m}\,\rangle},$$
where we denote by $\mathcal{D}iff^k$ the sheaf of differential operators of order $\leq k$.

Motivated by this, Kaneda and Ye defined in \cite{KY} the algebra
$$\bU^{(m)}\coloneqq \frac{T_{\bK}(\Di_{2p^m-1}(G))}{\langle \lambda-\lambda \epsilon_G,\,\,\delta\otimes\delta' - \delta'\otimes \delta - [\delta,\delta'],\,\,\,\delta\otimes\delta''-\delta\delta''\,\,\vert\,\,\lambda\in\bK,\,\delta''\in\Di_{p^m-1}(G),\,\,\delta,\delta'\in\Di_{p^m}(G)\,\rangle},$$
with $\epsilon_G$ the counit of $G$. They obtain, when $G$ is reductive, the following commutative diagram of $\bK[G]$-modules \cite[Cor 1.5]{KY}:
\begin{equation*} 
	\begin{CD}
		\bK[G]\otimes_\bK\bU^{(m)} @>{\sim}>> \Gamma(G,\mathcal{D}_{X/\bK}^{(m)})\\
		@VVV       @VVV \\
		\bK[G]\otimes_\bK\Di(G) @>{\sim}>> \Gamma(G,\mathcal{D}iff_{G/\bK})
	\end{CD}
\end{equation*}
with $\varinjlim\bU^{(m)}\cong \Di(G)$.

To be able to answer Parshall and Friedlander's question we need a slightly different construction of this algebra. We shall see that these constructions give isomorphic algebras in Section~\ref{sec3.3}.

\section{The Algebra $U^{[r]}(G)$}
\label{sec3}
\subsection{Filtered algebras} 
\label{sec3.1}
Before we get to the construction of the algebras $U^{[r]}(G)$ that we will be studying in this paper let us generalise slightly the situation we are considering, so that we can develop some notation and tools to work with in our particular circumstance. Suppose that $A$ is a filtered Hopf algebra $A=\cup_{k\in\bN}A_k$ with $A_0=\bK$ and such that the associated graded algebra $\gr(A)=\bigoplus_{k\in\bN}A_{k+1}/{A_{k}}$ is commutative (i.e. $[A_k,A_l]\subset A_{k+l-1}$ for all $k,l$). We shall denote $A_k^+\coloneqq A_k\cap\ker(\epsilon_A)$, where $\epsilon_A$ is the counit of $A$.

We can construct the following algebra.
$$U^{[k]}(A)\coloneqq\frac{T(A^+_k)}{Q_{k}},$$
where $Q_{k}$ is the ideal generated by the relations:

(i) $x\otimes y=xy$ if $x\in A^+_i$, $y\in A^+_j$ with $i+j<k+1$, and;

(ii) $x\otimes y-y\otimes x = [x,y]$ if $x\in A^+_i$, $y\in A^+_j$ with $i+j\leq k+1$.
\begin{dfn}
	Let $A$ be a filtered Hopf algebra $A=\cup_{k\in\bN}A_k$ satisfying the above conditions, and $B$ an associative $\bK$-algebra. We will call a $\bK$-linear map $\phi:A^{+}_k\to B$ an {\em indexed algebra subspace homomorphism} if $\phi(xy)=\phi(x)\phi(y)$ for all $x\in A^+_i$ and $y\in A^+_j $ with $i+j<k+1$, and $\phi([x,y])=[\phi(x),\phi(y)]$ for all $x\in A^+_i$ and $y\in A^+_j$ with $i+j\leq k+1$.
\end{dfn} 
There is a natural indexed algebra subspace homomorphism $\iota_Q:A^+_k\to U^{[k]}(A)$.
\begin{dfn}
	Let $A$ be a filtered Hopf algebra $A=\cup_{k\in\bN}A_k$  satisfying the above conditions. The {\em indexed algebra subspace dual} of $A_k^+$ is the set of all indexed algebra subspace homomorphisms from $A_k^+$ to $\bK$. We shall denote it by $(A_k^+)^{\overline{*}}$.
\end{dfn}
It is straightforward to prove the following universal property:
\begin{prop}\label{univ}
	Let $A$ be a filtered Hopf algebra $A=\cup_{k\in\bN}A_k$  satisfying the above conditions, and $B$ an associative $\bK$-algebra. Let $\phi:A^+_k\to B$ be an indexed algebra subspace homomorphism. Then there exists a unique algebra homomorphism $\overline{\phi}:U^{[k]}(A)\to B$ such that $\overline{\phi}\circ\iota_Q=\phi$.
\end{prop}
Let $\widehat{U}^{[k]}(A)$ be the algebra constructed in the same way as $U^{[k]}(A)$ except using $A_i$ instead of $A_i^{+}$  for $i\in\bN$ whenever relevant. This has a similar universal property, and using the universal properties for the linear maps $A_k^{+}\hookrightarrow A_k$ and $A_k\to \bK\oplus A_k^{+}$ it can be shown that the algebras $\widehat{U}^{[k]}(A)$ and $U^{[k]}(A)$ are isomorphic. We shall abuse notation to refer to both algebras as $U^{[k]}(A)$. [A similar argument can be made regarding the algebra $\bU^{(m)}$ defined in the Subsection~\ref{sec2.2}].
\begin{cor}\label{Hopf}
	Let $A$ be a filtered Hopf algebra $A=\cup_{k\in\bN}A_k$ satisfying the above conditions. Then $U^{[k]}(A)$ is a Hopf algebra for all $k\geq 0$. Furthermore, if $A$ is cocommutative then $U^{[k]}(A)$ is cocommutative.
\end{cor}
\begin{proof}
	We already know that $U^{[k]}(A)$ is an associative algebra. Applying Proposition~\ref{univ} to the comultiplication and counit maps on the coalgebra $A_k$ constructs the comultiplication and counit maps on $U^{[k]}(A)$. Furthermore, the antipode on $A$ sends $A_k$ to $A_k$ and so we get the antipode on $U^{[k]}(A)$ from Proposition~\ref{univ}. It is straightforward to check that the Hopf algebra axioms hold, and similarly straightforward to show cocommutativity when $A$ is cocommutative.
\end{proof}
\begin{dfn}
	Let $A$ be a filtered Hopf algebra $A=\cup_{k\in\bN}A_k$  satisfying the above conditions. An {\em indexed algebra subspace representation} of $A_k^{+}$ is an indexed algebra subspace homomorphism $\phi:A^+_k\to \End(M)$ where $M$ is a $\bK$-vector space.
\end{dfn}
\begin{dfn}
	Let $A$ be a filtered Hopf algebra $A=\cup_{k\in\bN}A_k$  satisfying the above conditions. A $\bK$-vector space $M$ is called an {\em indexed $A^+_k$-module} if there exists an indexed algebra subspace homomorphism $\theta:A_k^+\to\End(M)$. For $a\in A_k^+$ and $m\in M$ we shall often write $a\cdot m$ or just $am$ for the element $\theta(a)(m)$.
\end{dfn}
\begin{dfn}
	Let $A$ be a filtered Hopf algebra $A=\cup_{k\in\bN}A_k$  satisfying the above conditions, and let $(M_1,\theta_1),(M_2,\theta_2)$ be indexed $A_k^+$-modules. A {\em homomorphism of indexed $A_k^+$-modules} is a linear map $\phi:M_1\to M_2$ such that $\phi(am)=a\phi(m)$ for all $a\in A^+_k$ and $m\in M$.
\end{dfn}
We can use the universal property in a standard way to get the following theorem
\begin{prop}\label{bij}
	There is a bijection between the set of (isomorphism classes of) indexed $A^+_k$-modules and the set of (isomorphism classes of) $U^{[k]}(A)$-modules.
\end{prop}

\subsection{Higher Universal Enveloping Algebras}
\label{sec5}
Observe that, for an algebraic group $G$, the distribution algebra $\Di(G)$ is a filtered Hopf algebra $\Di(G)=\cup_{k\in\bN}\Di_k(G)$ with $\Di_0(G)=\bK$, such that the associated graded algebra $\gr(\Di(G))=\bigoplus_{k\in\bN}\Di_{k+1}(G)/{\Di_{k}(G)}$ is commutative. Furthermore, $\Di^+_k(G)$ is the same object as $\Di_k(G)^+$ and $\Di^+(G)$ is an ideal in $\Di(G)$.

We can now use the results of Section~\ref{sec3.1} to obtain analogues of the universal enveloping algebras. In particular, we define a higher universal enveloping algebra of $G$ of degree $r$ to be the algebra
$$U^{[r]}(G)\coloneqq U^{[p^{r+1}-1]}(\Di(G)).$$

The key observation which allows Parshall and Friedlander to develop and study their deformation algebras is that the $p$-th power map gives rise to a semilinear map $\xi:\fg\to Z(U(\fg))$ (i.e. for all $\alpha,\beta\in\bK$ and $x,y\in\fg$, $\xi(\alpha x + \beta y)=\alpha^p\xi(x)+\beta^p\xi(y)$). In order to make progress with the study of the structure of $U^{[r]}(G)$ we need to construct an analogue of the map $\xi$. We start with the following lemma. Note that when $\delta\in \Di_{k}^{+}(G)$ we already know that $\delta^p\in\Di_{pk}^{+}(G)$.
\begin{lemma}\label{powers}
	If $\delta\in\Di_{k}^{+}(G)$, then $\delta^p\in\Di_{pk-1}^{+}(G)$.
\end{lemma}
\begin{proof}
	Recall that $\bK[G]=\bK\oplus I_1$. Hence, for $m\in\bN$, $\bK[G]^{\otimes m}=\sum_{P_{i}\in\{\bK,I_1\}}P_1\otimes P_2\otimes\ldots\otimes P_m$. Using this and the counitary property of the Hopf algebra structure of $\bK[G]$, we have for $f\in I_1$,
	$$\Delta_{m-1}(f)\in f\otimes 1\otimes\ldots\otimes 1 + 1\otimes f \otimes\ldots\otimes 1 + \dots +
	1\otimes 1\otimes\ldots\otimes f + \sum _{\substack{a_i\in\{0,1\}\\2\leq\sum a_i\leq m}}I_1^{a_1}\otimes\ldots\otimes I_1^{a_m},$$
	where $\Delta_{m-1}$ is defined inductively by setting $\Delta_1$ as the comultiplication of $\bK[G]$ and $\Delta_l\coloneqq(\Delta_{l-1}\otimes \Id)\circ\Delta$. One can hence show by induction that for $f_1,\ldots,f_n\in I_1$, with $n\in\bN$, we have 
	$$\Delta_{m-1}(f_1\ldots f_n)\in\prod_{i=1}^n( f_i\otimes 1\otimes\ldots\otimes 1 + 1\otimes f_i \otimes\ldots\otimes 1 + \dots +
	1\otimes 1\otimes\ldots\otimes f_i) + \sum _{\substack{0\leq a_i\leq n\\n+1\leq\sum a_i\leq mn}}I_1^{a_1}\otimes\ldots\otimes I_1^{a_m}.$$
	Rewriting this slightly, we get
	\begin{multline*}
		\Delta_{m-1}(f_1\ldots f_n)\in\prod_{i=1}^n( f_i\otimes 1\otimes\ldots\otimes 1 + 1\otimes f_i \otimes\ldots\otimes 1 + \dots +
		1\otimes 1\otimes\ldots\otimes f_i) \\ + \sum_{j=1}^{m}\sum _{\substack{0\leq a_i\leq n\\n+1\leq\sum a_i\leq mn\\ a_j=0}}I_1^{a_1}\otimes\ldots\otimes I_1^{a_m} + \sum _{\substack{1\leq a_i\leq n\\\sum a_i=n+1}}I_1^{a_1}\otimes\ldots\otimes I_1^{a_m}.
	\end{multline*}	
	We now fix $m=p$ and $n=pk$. Given $\delta\in\Di_k^{+}(G)$ (so $\delta(I_1^{k+1})=0$ and $\delta(1)=0$) and $f_1,\ldots,f_{pk}\in I_1$ we have that
	\begin{multline*}
		\delta^p(f_1\ldots f_{pk})=(\delta\otimes\delta\otimes\ldots\otimes\delta)(\Delta_{p-1}(f_1\ldots f_{pk}))\in \\ (\delta\otimes\delta\otimes\ldots\otimes\delta)(\prod_{i=1}^{pk}( f_i\otimes 1\otimes\ldots\otimes 1 + 1\otimes f_i \otimes\ldots\otimes 1 + \dots +
		1\otimes 1\otimes\ldots\otimes f_i))\\
		+\sum_{j=1}^{p}\sum _{\substack{0\leq a_i\leq pk\\pk+1\leq\sum a_i\leq p^2k\\ a_j=0}}\delta(I_1^{a_1})\ldots\delta(I_1^{a_m})+\sum _{\substack{1\leq a_i\leq pk\\pk+1=\sum a_i}}\delta(I_1^{a_1})\ldots \delta(I_1^{a_p}).
	\end{multline*}
	Since $\delta(1)=0$, we get $$\sum_{j=1}^{p}\sum _{\substack{0\leq a_i\leq pk\\pk+1\leq\sum a_i\leq p^2k\\ a_j=0}}\delta(I_1^{a_1})\ldots\delta(I_1^{a_m})=0.$$
	Since $a_1+\ldots +a_p=pk+1$ implies $a_i\geq k+1$ for some $i$, and $\delta(I_1^{k+1})=0$, we also have 
	$$\sum _{\substack{1\leq a_i\leq pk\\pk+1=\sum a_i}}\delta(I_1^{a_1})\ldots \delta(I_1^{a_p})=0.$$
	Now, we want to compute $(\delta\otimes\delta\otimes\ldots\otimes\delta)(\prod_{i=1}^{pk}( f_i\otimes 1\otimes\ldots\otimes 1 + 1\otimes f_i \otimes\ldots\otimes 1 + \dots +
	1\otimes 1\otimes\ldots\otimes f_i))$.
	
	Observe that $$\prod_{i=1}^{pk}( f_i\otimes 1\otimes\ldots\otimes 1 + 1\otimes f_i \otimes\ldots\otimes 1 + \dots +
	1\otimes 1\otimes\ldots\otimes f_i)=\sum f_{A_1}\otimes\ldots\otimes f_{A_p},$$
	where the sum is over all ordered partitions $A_1,\ldots,A_p$ of the set $\{1,\ldots,pk\}$ where the sets can be empty (ordered partition meaning for example that $\{1,2\},\{3,4\}$ is different from $\{3,4\},\{1,2\}$), and where, if $A_i=\{j_1,\ldots,j_s\}$ with $j_1<\ldots<j_s$, we denote $f_{A_i}=f_{j_1}f_{j_2}\ldots f_{j_s}$. Then 
	$$
	(\delta\otimes\delta\otimes\ldots\otimes\delta)(\prod_{i=1}^{pk}( f_i\otimes 1\otimes\ldots\otimes 1 + 1\otimes f_i \otimes\ldots\otimes 1 + \dots +
	1\otimes 1\otimes\ldots\otimes f_i))
	=\sum \delta(f_{A_1})\ldots\delta(f_{A_p})
	$$	
	where the sum is over the same set as before.
	
	For ordered partitions containing empty sets, $\delta(f_{A_i})=\delta(1)=0$ for those $i$ with $A_i=\emptyset$. Furthermore, if two ordered partitions containing no empty sets are rearrangements of each other, they give the same summand in the above sum since $\bK$ is a field. In particular, there are $p!$ such partitions which give the same summand, so this summand appears $p!$ times. Hence 
	$$
	(\delta\otimes\delta\otimes\ldots\otimes\delta)(\prod_{i=1}^{pk}( f_i\otimes 1\otimes\ldots\otimes 1 + 1\otimes f_i \otimes\ldots\otimes 1 + \dots +
	1\otimes 1\otimes\ldots\otimes f_i))
	=\sum p!\delta(f_{A_1})\ldots\delta(f_{A_p})=0
	$$
	where this time the second sum is over {\em unordered partitions} with $p$ non-empty sets in them. 
	
	Hence, we have that $\delta^p(f_1\ldots f_{pk})=0$. That is to say, $\delta^p\in\Di_{pk-1}^{+}(G)$.
\end{proof}

In particular, if $\delta\in\Dipri$ then $\delta^p\in\Diprj$. This allows us to define a map $\xi_r:\Di_{p^r}^+(G)\rightarrow U^{[r]}(G)$ as $\xi_r(\delta)=\delta^{\otimes p}-\delta^p$ where the first exponent is in $U^{[r]}(G)$ and the second is in $\Di(G)$.

\begin{lemma}\label{semi}
	$\xi_r$ is semilinear.
\end{lemma}
\begin{proof}
	Clearly $\xi_r(\lambda\delta)=\lambda^p\xi_r(\delta)$ if $\lambda\in\bK$ and $\delta\in\Dipri$. We now want to show $\xi_r(\mu+\rho)=\xi_r(\mu)+\xi_r(\rho)$ for $\mu,\rho\in \Di^{+}_{p^{r}}(G)$. Observe that, by definition,	
	$$\xi_r(\mu+\rho)=(\mu+\rho)^{\otimes p}-(\mu+\rho)^p.$$
	We have that
	$$(\mu+\rho)^{\otimes p}=\sum_{a_i\in\{0,1\}}\eta_{a_1}\otimes\ldots\otimes\eta_{a_p},$$ where $\eta_0=\mu$ and $\eta_1=\rho$. Applying $\mu\otimes\rho-\rho\otimes\mu=[\mu,\rho]\in \Di^{+}_{2p^{r}-1}(G)$, we get 
	$$(\mu+\rho)^{\otimes p}=\sum_{i=0}^p\binom{p}{i}\mu^{\otimes i}\otimes\rho^{\otimes (p-i)} - \Psi$$ where $\Psi$ is a sum of terms in $\Upr$, each of which is the tensor product of elements of $\Di(G)$ where the sum of the grades is less than $p^{r+1}$. Hence, $\Psi$ is obtained from the product of these elements in $\Di(G)$, by the definition of $\Upr$. Since $\Char \bK= p$, we get 
	$$(\mu+\rho)^{\otimes p}=\mu^{\otimes p} + \rho^{\otimes p} - \Psi.$$
	Similarly, 
	$$(\mu+\rho)^{p}=\sum_{a_i\in\{0,1\}}\eta_{a_1}\ldots\eta_{a_p},$$ where $\eta_0=\mu$ and $\eta_1=\rho$. Applying $\mu\rho-\rho\mu=[\mu,\rho]\in \Di_{2p^{r}-1}(G)$, we get $$(\mu+\rho)^{p}=\sum_{i=0}^p\binom{p}{i}\mu^{i}\rho^{p-i} - \Psi$$ where $\Psi$ is exactly the same $\Psi$ as above since the multiplication in the expression of $\Psi$ is the same in $\Di(G)$ and $U^{[r]}(G)$.
	So $$(\mu+\rho)^{p}=\mu^{p} + \rho^{p} - \Psi.$$
	Hence $\xi_r(\mu+\rho)=\xi_r(\mu)+\xi_r(\rho)$
\end{proof}

For $k\leq r$, define $X_{p^k}$ to be the $\bK$-span of $\{\mu\in\Di^{+}_{p^k}(G)\,\vert\,\mu=\rho_1\rho_2\,\mbox{for}\,\rho_i\in\Di_{j_i}(G)\,\mbox{with}\,j_1+j_2\leq p^k \,\mbox{and}\,j_1,j_2<p^k\}\subset \Upr$. Define $Y_{p^k}$ to be a vector space complement of this subspace in $\Di^{+}_{p^k}(G)$; when $G$ is reductive, we take it to be the one with basis $\{\ve_{\alpha}^{(p^k)},\binom{\vh_t}{p^k}\,\vert\,\alpha\in\Phi,1\leq t\leq d\}$ (see Subsection~\ref{sec3.2} for the notation). The next proposition shows that $\xi_r$ is only non-trivial on the subspace $Y_{p^r}$. 

\begin{prop}\label{zero}
	For all $0\leq k\leq r$, $\xi_r(X_{p^k})=0$.
\end{prop}
\begin{proof}
	Since $X_{p^k}\subset X_{p^r}$ for all $0\leq k\leq r$, it is sufficient to prove that $\xi_r(X_{p^r})=0$.
	
	Suppose $\mu\in \Di_i(G)$, $\rho\in \Di_j(G)$, where $i+j\leq p^r$ and $i,j>0$. So $\mu\rho\in \Di_{p^r}(G)$. Consider $\xi_r(\mu\rho)=(\mu\rho)^{\otimes p}-(\mu\rho)^p$. Note that $\mu\rho-\mu\otimes\rho=0$ as $i+j\leq p^{r}<p^{r+1}$. We have $$(\mu\rho)^{\otimes p}=\mu\otimes(\rho\otimes\mu)\otimes\ldots\otimes(\rho\otimes\mu)\otimes\rho.$$
	Furthermore $\rho\otimes\mu-\mu\otimes\rho=[\rho,\mu]\in \Di_{p^{r}-1}(G)$. Hence 	
	$$(\mu\rho)^{\otimes p}=\mu^{\otimes p}\otimes\rho^{\otimes p}-\Phi,$$ where $\Phi$ is a sum of terms in $\Upr$, each of which is the tensor product of elements of $\Di(G)$ where the sum of the grades is less that $p^{r+1}$. Hence, $\Phi$ is obtained from the product of these elements in $\Di(G)$.
	Similarly, we have $$(\mu\rho)^p=\mu(\rho\mu)\ldots(\rho\mu)\rho.$$
	Since $\rho\mu-\mu\rho=[\rho,\mu]$ by definition, we get that
	$$(\mu\rho)^p=\mu^p\rho^p -\Phi,$$ where $\Phi$ is exactly the same as above, since it doesn't matter when calculating $\Phi$ if the multiplication is done in $\Di(G)$ or in $U^{[r]}(G)$ because of the grades of the elements being multiplied.
	
	Hence, $\xi_r(\mu\rho)=(\mu\rho)^{\otimes p}-(\mu\rho)^p=\mu^{\otimes p}\otimes\rho^{\otimes p}-\mu^p\rho^p$. Since $\mu\in \Di_i(G)$ and $i<p^{r}$, we have $\mu^{\otimes p}=\mu^p$, and similarly for $\rho$. So $\xi_r(\mu\rho)=\mu^p\otimes\rho^p-\mu^p\rho^p$. Furthermore, $\mu^p\in \Di_{pi-1}(G)$ and $\rho^p\in \Di_{pj-1}(G)$, so $\mu^p\otimes\rho^p=\mu^p\rho^p$, so $\xi_r(\mu\rho)=0$.
\end{proof}

At this point, we would like to show that the image of $\xi_r$ is central in $\Upr$. However, the proofs of this result which are known to the author for universal enveloping algebras do not appear to work in this case. For example, a priori there is no reason why $\ad(\delta^p)=\ad(\delta)^p$ should hold when $r\neq 0$. Hence, to progress further we must move to the case of reductive groups where calculations can be made more explicit.

\subsection{Reductive groups}
\label{sec3.2}

From now on, unless specified otherwise, $G$ will be a reductive algebraic group over an algebraically closed field $\bK$ of characteristic $p>0$.

Let $B$ be a Borel subgroup of $G$ with maximal torus $T$ and unipotent radical $U$. Say that $T\cong(\bG_m)^d$, where $\bG_m$ is the multiplicative group, and let $X(T)=\Hom(T,\bG_m)$ be the group of characters of $T$. Let $\Phi$ be the root system of $G$ with respect to $T$. We specify a set of positive roots $\Phi^{+}$ with corresponding simple roots $\Pi=\{\alpha_1,\ldots,\alpha_n\}$ and we fix an ordering on $\Phi$.

Denote $\fg=\Lie(G)$, $\fb=\Lie(B)$, $\fh=\Lie(T)$ and $\fn^{+}=\Lie(U)$. As in \cite[II.1.11]{Jan}, $\fg$ has a basis $\{\ve_\alpha, \vh_t\,;\,\alpha\in\Phi,1\leq t\leq d\}$ (where $\ve_\alpha=X_\alpha\otimes 1$ and $\vh_t=H_t\otimes 1$ in Jantzen's notation). We set $\vh_\alpha=[\ve_{\alpha},\ve_{-\alpha}]$. Throughout this paper we shall abuse notation by using the same symbols $\ve_\alpha$ and $\vh_t$ for the corresponding elements over any base ring. One may see this abuse, for example, in the following statement: the elements $\ve_\alpha\in\fg_{\bC}$ for $\alpha\in\Phi$ form a Chevalley system in $\fg_{\bC}$, where a Chevalley system is as defined in \cite[ch. VIII, \textsection 12]{Bou}. Here, $\fg_{\bC}$ is the complex reductive Lie algebra corresponding to $\fg$ over the field $\bC$.

Let us recall the construction of the standard bases for the universal enveloping algebra $U(\fg)$ and the distribution algebra $\Di(G)$. In both cases we start by considering the complex reductive Lie algebra $\fg_{\bC}$, and we look at elements in the universal enveloping algebra $U(\fg_{\bC})$. 
Recall that $U(\fg_{\bC})$ has $\bC$-basis
$$\{\prod_{\alpha\in\Phi^+}\ve_{\alpha}^{i_\alpha}\prod_{t=1}^d\vh_{t}^{k_t}\prod_{\alpha\in\Phi^+}\ve_{-\alpha}^{j_\alpha}\quad ;\quad 0\leq i_\alpha,j_\alpha,k_t\}.$$
We then look at the following $\bZ$-forms in $U(\fg_\bC)$: 
$$U(\fg)_{\bZ}=\bZ\{\prod_{\alpha\in\Phi^+}\ve_{\alpha}^{i_\alpha}\prod_{t=1}^d\vh_{t}^{k_t}\prod_{\alpha\in\Phi^+}\ve_{-\alpha}^{j_\alpha}\quad ;\quad 0\leq i_\alpha,j_\alpha,k_t\},$$
$$\widetilde{U}(\fg)_\bZ=\bZ\{\prod_{\alpha\in\Phi^+}\ve_{\alpha}^{(i_\alpha)}\prod_{t=1}^d\binom{\vh_{t}}{k_t}\prod_{\alpha\in\Phi^+}\ve_{-\alpha}^{(j_\alpha)}\quad ;\quad 0\leq i_\alpha,j_\alpha,k_t\}$$
where $\ve_{\alpha}^{(i_\alpha)}\coloneqq\frac{\ve_{\alpha}^{i_\alpha}}{i_\alpha!}$ and $\binom{\vh_{t}}{k_t}\coloneqq\frac{\vh_t(\vh_t-1)\ldots(\vh_t-k_t+1)}{k_t!}$. We call $\ve_{\alpha}^{(i_\alpha)}$ and $\binom{\vh_{t}}{k_t}$ \emph{divided powers} of $\ve_\alpha$ and $\vh_t$.

It is easy to see that the first of these is a $\bZ$-form from the definitions of the commutators, while the fact that the second is a $\bZ$-form was proved by Kostant in \cite{Kos} in the case when $G$ is semisimple and simply-connected -- the more general result can be found in Jantzen \cite[II.1.12]{Jan}. From this, we get $U(\fg)=U(\fg)_\bZ\otimes_\bZ\bK$ and $\Di(G)=\widetilde{U}(\fg)_\bZ\otimes_\bZ\bK$. 

To get a basis for the algebra $U^{[r]}(G)$ we shall apply the same process with a $\bZ_{(p)}$-form. Recall here that $\bZ_{(p)}=\{\frac{a}{b}\in\bQ\,\vert\,\hcf(a,b)=1,\, p\nmid b\}$ is a commutative local ring.

Given an integer $M=a_0+a_1p+\ldots+a_{r}p^{r}$ where $0\leq a_0,\ldots,a_{r-1}<p$ and $a_{r}\geq 0$, we shall define $$\ve_\alpha^{\llbracket M\rrbracket}=\ve_\alpha^{a_0}(\ve_\alpha^{(p)})^{a_1}\ldots (\ve_\alpha^{(p^{r})})^{a_{r}}\in U(\fg_\bC)$$ for $\alpha\in\Phi$. Furthermore, define  $$\binom{\vh_t}{\llbracket M\rrbracket}=\binom{\vh_t}{1}^{a_0}\binom{\vh_t}{p}^{a_1}\ldots\binom{\vh_t}{p^{r}}^{a_{r}}\in U(\fg_{\bC})$$ for $1\leq t\leq d$.

\begin{prop}
	The subset 
	$$U^{\llbracket r\rrbracket}(\fg)_{\bZ_{(p)}}\coloneqq\bZ_{(p)}\{\prod_{\alpha\in\Phi^+}\ve_{\alpha}^{\llbracket i_\alpha\rrbracket}\prod_{t=1}^d\binom{\vh_{t}}{\llbracket k_t\rrbracket}\prod_{\alpha\in\Phi^+}\ve_{-\alpha}^{\llbracket j_\alpha\rrbracket }\quad ;\quad 0\leq i_\alpha,j_\alpha,k_t\}\subset U(\fg_{\bC})$$
	is a well-defined $\bZ_{(p)}$-form of $U(\fg_{\bC})$.
	
\end{prop}

\begin{proof}
	For this to be well defined, we need to show that it closed under multiplication. It is clearly enough to show that certain commutators lie inside $U^{\llbracket r\rrbracket}(\fg)_{\bZ_{(p)}}$. Let us introduce the notation $$\widetilde{U}^{\llbracket r\rrbracket}(\fg)_{\bZ_{(p)}}\coloneqq\bZ_{(p)}\{\prod_{\alpha\in\Phi^+}\ve_{\alpha}^{(i_\alpha)}\prod_{t=1}^d\binom{\vh_{t}}{k_t}\prod_{\alpha\in\Phi^+}\ve_{-\alpha}^{(j_\alpha)}\, ;\, 0\leq i_\alpha,j_\alpha,k_t<p^{r+1}\},$$ which lies inside $\widetilde{U}(\fg)_{\bZ_{(p)}}\cap U^{\llbracket r\rrbracket}(\fg)_{\bZ_{(p)}}$.
	
	One can now compute that, for $\alpha,\beta\in\Phi$, $1\leq t,t_1,t_2\leq d$ and $0\leq s,u<r+1$, we have
	$$[\ve_\alpha^{(p^s)},\ve_\beta^{(p^u)}]\in \widetilde{U}^{\llbracket r\rrbracket}(\fg)_{\bZ_{(p)}},$$
	$$[\ve_\alpha^{(p^s)},\ve_{-\alpha}^{(p^u)}]\in \widetilde{U}^{\llbracket r\rrbracket}(\fg)_{\bZ_{(p)}},$$
	$$[\ve_\alpha^{(p^s)},\binom{\vh_t}{p^u}]=\sum_{l=0}^{p^u-1}\binom{-\alpha(\vh_t)p^s}{p^u-l}\binom{\vh_t}{l}\ve_{\alpha}^{(p^s)}\in \widetilde{U}^{\llbracket r\rrbracket}(\fg)_{\bZ_{(p)}},$$
	$$[\binom{\vh_{t_1}}{p^s},\binom{\vh_{t_2}}{p^u}]=0.$$	
	More specifically, we know that when we write these commutators in the divided powers basis we have coefficients in $\bZ_{(p)}$ (this just follows from $\widetilde{U}(\fg)_{\bZ_{(p)}}$ being a $\bZ_{(p)}$-form). Hence, all we have to show is that none of the divided power indices exceed $p^{r+1}-1$. The first two of these calculations can be checked directly using \cite{Kos} and \cite{Cham}, while the second two are clear. For example, if $\{\alpha,\beta\}$ form the fundamental roots for a root system of type $G_2$ with $\beta$ the long root, then we have 
	$$
	[\ve_\alpha^{(p^s)},\ve_\beta^{(p^u)}]=\sum \epsilon_{k_1,k_2,k_3,k_4}\ve_\beta^{(p^u-k_1-k_2-k_3-2k_4)}(\prod_{j=1}^{3}\ve_{j\alpha+\beta}^{(k_j)}) \ve_{3\alpha+2\beta}^{(k_4)}\ve_\alpha^{(p^s-k_1-2k_2-3k_3-3k_4)}
	$$
	where the sum is over all $k_1,k_2,k_3,k_4\geq 0$, not all zero, such that $k_1+k_2+k_3+2k_4\leq p^s$ and $k_1+2k_2+3k_3+3k_4\leq p^u$ and $\epsilon_{k_1,k_2,k_3,k_4}\in\{1,-1\}$ for all $k_1,k_2,k_3,k_4$. In particular, none of the heights of the divided powers are greater than or equal to $p^{r+1}$. The rest are similar.
\end{proof}

We can hence form $U^{\llbracket r\rrbracket}(\fg)\coloneqq U^{\llbracket r\rrbracket}(\fg)_{\bZ_{(p)}}\otimes_{\bZ_{(p)}}\bK$. 

\begin{prop}\label{basis}
	There is an isomorphism of algebras $U^{\llbracket r\rrbracket}(\fg)\cong U^{[r]}(G)$.
\end{prop}

\begin{proof}
	
	We prove this by constructing an algebra homomorphism $U^{[r]}(G)\rightarrow U^{\llbracket r\rrbracket}(\fg)$ using the universal property and showing that it sends a basis of $U^{[r]}(G)$ to a basis of $U^{\llbracket r\rrbracket}(\fg)$. 
	
	$\Di_{p^{r+1}-1}(G)$ has $\bK$-basis $$\{\prod_{\alpha\in\Phi^+}\ve_{\alpha}^{(i_\alpha)}\prod_{t=1}^d\binom{\vh_{t}}{k_t}\prod_{\alpha\in\Phi^+}\ve_{-\alpha}^{(j_\alpha)}\, :\, \sum_{\alpha\in\Phi^+}(i_\alpha+j_\alpha) + \sum_{t=1}^d k_t<p^{r+1}\}.$$ Define $\phi:\Di_{p^{r+1}-1}(G)\rightarrow U^{\llbracket r\rrbracket}(\fg)$ by $$\phi(\prod_{\alpha\in\Phi^+}\ve_{\alpha}^{(i_\alpha)}\prod_{t=1}^d\binom{\vh_{t}}{k_t}\prod_{\alpha\in\Phi^+}\ve_{-\alpha}^{(j_\alpha)})=\prod_{\alpha\in\Phi^+}\ve_{\alpha}^{(i_\alpha)}\prod_{t=1}^d\binom{\vh_{t}}{k_t}\prod_{\alpha\in\Phi^+}\ve_{-\alpha}^{(j_\alpha)}.$$ The fact that $\phi(\delta\rho)=\phi(\delta)\phi(\rho)$ if $\delta\in \Di^{+}_i(G)$, $\rho\in \Di^{+}_j(G)$ with $i+j<p^{r+1}$ and $\phi([\delta,\rho])=[\phi(\delta),\phi(\rho)]$ if $\delta\in \Di^{+}_i(G)$, $\rho\in \Di^{+}_j(G)$ with $i+j\leq p^{r+1}$ is obvious from how basis elements in $\Di_{p^{r+1}-1}(G)$ multiply (since below the $p^{r+1}$ level, the multiplication is the same in  $U^{\llbracket r\rrbracket}(\fg)$ and $\Di(G)$). Hence we get an algebra homomorphism $\phi:U^{[r]}(G)\rightarrow U^{\llbracket r\rrbracket}(\fg)$ from the universal property.
	
	We now need some notation for the elements in $U^{[r]}(G)$. Given an integer $M=a_0+a_1p+\ldots+a_{r}p^{r}$ where $0\leq a_0,\ldots,a_{r-1}<p$ and $a_{r}\geq 0$, we shall define $$\ve_\alpha^{\llbracket M\rrbracket_\otimes}=\ve_\alpha^{\otimes a_0}\otimes(\ve_\alpha^{(p)})^{\otimes a_1}\otimes \ldots\otimes (\ve_\alpha^{(p^{r})})^{\otimes a_{r}}\in U^{[r]}(G)$$ for $\alpha\in\Phi$. Furthermore, define  $$\binom{\vh_t}{\llbracket M\rrbracket_\otimes}=\binom{\vh_t}{1}^{\otimes a_0}\otimes\binom{\vh_t}{p}^{\otimes a_1}\otimes\ldots\otimes\binom{\vh_t}{p^{r}}^{\otimes a_{r}}\in U^{[r]}(G)$$ for $1\leq t\leq d$. Then $$\phi(\bigotimes_{\alpha\in\Phi^+}\ve_{\alpha}^{\llbracket i_\alpha\rrbracket_\otimes}\bigotimes_{t=1}^d\binom{\vh_{t}}{\llbracket k_t\rrbracket_\otimes}\bigotimes_{\alpha\in\Phi^+}\ve_{-\alpha}^{\llbracket j_\alpha\rrbracket_\otimes })=\prod_{\alpha\in\Phi^+}\ve_{\alpha}^{\llbracket i_\alpha\rrbracket}\prod_{t=1}^d\binom{\vh_{t}}{\llbracket k_t \rrbracket}\prod_{\alpha\in\Phi^+}\ve_{-\alpha}^{\llbracket j_\alpha\rrbracket }.$$ 
	
	Furthermore, it is not difficult to see that the $\bigotimes_{\alpha\in\Phi^+}\ve_{\alpha}^{\llbracket i_\alpha\rrbracket_\otimes}\bigotimes_{t=1}^d\binom{\vh_{t}}{\llbracket k_t \rrbracket_\otimes}\bigotimes_{\alpha\in\Phi^+}\ve_{-\alpha}^{\llbracket j_\alpha\rrbracket_\otimes }$, for $i_\alpha,j_{-\alpha},k_t\in\bN$, span $U^{[r]}(G)$ as a vector space. They are also linearly independent, since their images under the map $\phi$ are. Thus, $\phi$ maps a basis to a basis, and the result holds.
	
	%
\end{proof}

Hence $U^{\llbracket r\rrbracket}(\fg)\cong U^{[r]}(G)$ as algebras and $U^{[r]}(G)$ has the desired basis, which we shall generally abuse notation to denote it as $ \{\prod_{\alpha\in\Phi^+}\ve_{\alpha}^{\llbracket i_\alpha\rrbracket}\prod_{t=1}^d\binom{\vh_{t}}{\llbracket k_t \rrbracket}\prod_{\alpha\in\Phi^+}\ve_{-\alpha}^{\llbracket j_\alpha\rrbracket }\quad :\quad 0\leq i_\alpha,j_\alpha,k_t\}$. 

Note that the universal property of $U(\fg)$ gives a $\bK$-algebra homomorphism $U(\fg)\to U^{[0]}(G)$. This basis guarantees that this is an isomorphism of $\bK$-algebras (in fact, of Hopf algebras, by considering the effect of the comultiplication, counit and antipode on the corresponding bases).  Hence, the representation theory of reductive Lie algebras over a field of characteristic $p>0$ as studied by Friedlander and Parshall in \cite{PF2} and \cite{PF1} exists within our theory as the case when $r=0$. One can also see this using Kaneda and Ye's construction  $\bU^{(0)}$ and Proposition~\ref{KYisom} below.

With a basis of $\Upr$ in place, we can now prove the following proposition.

\begin{prop}\label{cent}
	If $G$ is reductive, the image of $\xi_r$ is central in $\Upr$. 
\end{prop}
\begin{proof}
	
	%
	By Propositions \ref{semi} and \ref{zero}, it is enough to show that $\xi_r(\ve_\alpha^{(p^r)})$ and $\xi_r(\binom{\vh_t}{p^r})$ are central for $\alpha\in\Phi$ and $1\leq t\leq d$. We know that $\xi_r(\ve_\alpha^{(p^r)})=(\ve_\alpha^{(p^r)})^{\otimes p}$ and $\xi_r(\binom{\vh_t}{p^r})=\binom{\vh_t}{p^r}^{\otimes p} - \binom{\vh_t}{p^r}$. By the given basis of $\Upr$, it is enough to show that $\xi_r(\ve_\alpha^{(p^r)})$ and $\xi_r(\binom{\vh_t}{p^r})$ commute with each element of $\Dipri$.
	
	Observe that in the notation coming from the $\bZ_{(p)}$-form the multiplicative notation means the tensor product notation in $U^{[r]}(G)$. This gives us that for $\alpha,\beta\in \Phi$ with $\alpha\neq -\beta$ and $0<s\leq r$, Lemma 8 in \cite{Cham} shows
	$$[(\ve_\alpha^{(p^r)})^{p},\ve_\beta^{(p^s)}]=\frac{p^{r+1}!}{(p^r!)^p}[\ve_\alpha^{(p^{r+1})},\ve_\beta^{(p^s)}]\in \frac{p^{r+1}!}{(p^r!)^p} U^{\llbracket r\rrbracket}(\fg)_{\bZ_{(p)}},$$	
	$$[(\ve_\alpha^{(p^r)})^{p},\ve_{-\alpha}^{(p^s)}]=\frac{p^{r+1}!}{(p^r!)^p}[\ve_\alpha^{(p^{r+1})},\ve_{-\alpha}^{(p^s)}]\in \frac{p^{r+1}!}{(p^r!)^p} U^{\llbracket r\rrbracket}(\fg)_{\bZ_{(p)}}.$$
	
	
	In fact, the equations from \cite[Lemma 8]{Cham} show that these commutators lie in $\frac{p^{r+1}!}{(p^r!)^p}\widetilde{U}^{\llbracket r\rrbracket}(\fg)_{\bZ_{(p)}}$, not just in $\frac{p^{r+1}!}{(p^r!)^p}U^{\llbracket r\rrbracket}(\fg)_{\bZ_{(p)}}$. 
	The reader can see this with the observation that if, for example,  $\{\alpha,\beta\}$ form the fundamental roots for a root system of type $G_2$ with $\beta$ the long root, then we have as in \cite[Lemma 8]{Cham} that 
	$$
	[\ve_\alpha^{(p^{r+1})},\ve_\beta^{(p^s)}]=\sum \epsilon_{k_1,k_2,k_3,k_4}\ve_\beta^{(p^s-k_1-k_2-k_3-2k_4)}(\prod_{j=1}^{3}\ve_{j\alpha+\beta}^{(k_j)}) \ve_{3\alpha+2\beta}^{(k_4)}\ve_\alpha^{(p^{r+1}-k_1-2k_2-3k_3-3k_4)}
	$$	
	where the sum is over all $k_1,k_2,k_3,k_4\geq 0$, not all zero, such that $k_1+k_2+k_3+2k_4\leq p^{r+1}$ and $k_1+2k_2+3k_3+3k_4\leq p^s$ and $\epsilon_{k_1,k_2,k_3,k_4}\in\{1,-1\}$ for all $k_1,k_2,k_3,k_4$. In particular, none of the divided powers are greater than or equal to $p^{r+1}$.
	
	Since $\frac{p^{r+1}!}{(p^r!)^p}\in\bZ$ vanishes modulo $p$, the above equations hence show that the commutators vanish in $U^{\llbracket r\rrbracket}(\fg)= U^{\llbracket r\rrbracket}(\fg)_{\bZ_{(p)}}\otimes_{\bZ_{(p)}}\bK$.
	
	Furthermore, 
	$$[(\ve_\alpha^{(p^r)})^{p},\binom{\vh_t}{p^s}]=\sum_{l=0}^{p^s-1}\binom{-\alpha(\vh_t)p^{r+1}}{p^s-l}\binom{\vh_t}{l}(\ve_{\alpha}^{(p^{r})})^{p}=0,$$
	where the last equality follows from the observation that $\binom{-\alpha(\vh_t)p^{r+1}}{p^s-l}=0$ modulo $p$ for all $0\leq l\leq p^s-1$. This comes from Lucas' Theorem and the fact that $s<r+1$. This gives the centrality of $\xi_r(\ve_\alpha^{(p^r)})$. 
	For $\xi_r(\binom{\vh_t}{p^r})$ we have
	$$[\binom{\vh_t}{p^r}^{\otimes p}-\binom{\vh_t}{p^r},\binom{\vh_u}{p^s}]=0$$
	and 
	\begin{align*}
		(\binom{\vh_t}{p^r}^{\otimes p}-\binom{\vh_t}{p^r})\ve_\alpha^{(p^s)}=\ve_\alpha^{(p^s)}(\binom{\vh_t -\alpha(\vh_t)p^s}{p^r}^{\otimes p}-\binom{\vh_t -\alpha(\vh_t)p^s}{p^r}) \\ =\ve_\alpha^{(p^s)}((\sum_{l=0}^{p^r}\binom{\vh_t}{l}\binom{-\alpha(\vh_t)p^s}{p^r-l})^{\otimes p}-\sum_{l=0}^{p^r}\binom{\vh_t}{l}\binom{-\alpha(\vh_t)p^s}{p^r-l}) \\
		=\ve_\alpha^{(p^s)}(\sum_{l=0}^{p^r}\binom{\vh_t}{l}^{\otimes p}\binom{-\alpha(\vh_t)p^s}{p^r-l}-\sum_{l=0}^{p^r}\binom{\vh_t}{l}\binom{-\alpha(\vh_t)p^s}{p^r-l})\\ =\ve_\alpha^{(p^s)}(\binom{\vh_t}{p^r}^{\otimes p}-\binom{\vh_t}{p^r})
	\end{align*}	
	since $\binom{\vh_t}{l}^{\otimes p}=\binom{\vh_t}{l}$ for $l<p^r$. This gives the centrality of $\xi_r(\binom{\vh_t}{p^r})$. Hence the image of $\xi_r$ is central.

\end{proof}

\subsection{Centres}
\label{sec5.2}
As always from now on, $G$ is reductive. Let $Z_r(G)$ be the subalgebra of $Z(\Upr)$ generated by the $\xi_r(\delta)$ for $\delta\in \Dipri$. Using Propositions~\ref{basis} and \ref{cent}, we can easily see that $Z_r(G)$ is generated by $(\ve_\alpha^{(p^r)})^{\otimes p}$ for $\alpha\in \Phi$ and $\binom{\vh_t}{p^r}^{\otimes p}-\binom{\vh_t}{p^r}$ for $1\leq t\leq d$. From Proposition~\ref{basis}, it is clear that these elements are algebraically independent over $\bK$.

Note the semilinearity of $\xi_r$ induces an algebra homomorphism from $S(Y_{p^r}^{(1)})$ (the symmetric algebra on the vector space $Y_{p^r}^{(1)}$ defined above) to $Z_r(G)$. This map is bijective.

As a $Z_r(G)$-module under left multiplication, $\Upr$ is free of rank $p^{(r+1)\dim(\fg)}$ with basis
$$\{\prod_{\alpha\in\Phi^+}\ve_{\alpha}^{\llbracket i_\alpha\rrbracket}\prod_{t=1}^d\binom{\vh_{t}}{\llbracket k_t\rrbracket}\prod_{\alpha\in\Phi^+}\ve_{-\alpha}^{\llbracket j_\alpha\rrbracket }\quad ;\quad 0\leq i_\alpha,j_\alpha,k_t<p^{r+1}\}.$$
This leads us to the following proposition.

\begin{prop}\label{fingen}
	The centre $Z(\Upr)$ of $\Upr$ is a finitely generated algebra over $\bK$. As a $Z(\Upr)$-module, $\Upr$ is finitely generated.
\end{prop}

\begin{theorem}\label{findim}
	Let $E$ be an irreducible $\Upr$-module. Then $E$ is finite-dimensional, of dimension less than or equal to $p^{(r+1)\dim(\fg)}$.
\end{theorem}

\begin{proof}
	This follows in exactly the same way as Theorem A.4 in \cite{Jan3}.
\end{proof}

\subsection{Comparison with Kaneda-Ye Construction}
\label{sec3.3}

Recall that Kaneda and Ye in \cite{KY} construct the algebra $\bU^{(r)}$ as 
$$\bU^{(r)}\coloneqq \frac{T_{\bK}(\Di_{2p^r-1}(G))}{\langle \lambda-\lambda \epsilon_G,\,\,\delta\otimes\delta' - \delta'\otimes \delta - [\delta,\delta'],\,\,\,\delta\otimes\delta''-\delta\delta''\,\,\vert\,\,\lambda\in\bK,\,\delta''\in\Di_{p^r-1}(G),\,\,\delta,\delta'\in\Di_{p^r}(G)\,\rangle},$$
with $\epsilon_G$ the counit of $G$. 

\begin{prop}\label{KYisom}
	The algebras $\bU^{(r)}$ and $\Upr$ are isomorphic.
\end{prop}

\begin{proof}
	
	The algebra $\bU^{(r)}$ has a clear universal property, which causes the inclusion $\Di_{2p^r-1}(G)\hookrightarrow\Upr$ to induce an algebra homomorphism $\bU^{(r)}\to \Upr$. The surjectivity of this homomorphism is obvious from the basis constructed in Section~\ref{sec3.2}.
	
	It is left as an exercise for the reader to show that the proof of Proposition~\ref{basis}, showing that the algebra $\Upr$ has the given basis, applies equally well to the algebra $\bU^{(r)}$. This guarantees that the algebra homomorphism $\bU^{(r)}\to \Upr$ is an isomorphism.
\end{proof}

\section{Connection with other algebras}
\label{sec4}
\subsection{Universal enveloping algebra}
\label{sec4.1}
%

Recalling that reductive algebraic groups are defined over $\bF_p$, we may consider the Frobenius kernel $G_{(s)}$ ($s\in\bN$) as the kernel of the \emph{geometric Frobenius endomorphism} $F_g^s:G\to G$, i.e. the endomorphism of $G$ corresponding to the Hopf algebra homomorphism $\bF_p[G]\otimes_{\bF_p}\bK\to \bF_p[G]\otimes_{\bF_p}\bK$ which sends $f\otimes a$ to $f^{p^s}\otimes a$. Applying the distribution functor to $F_g^s$, we get a Hopf algebra homomorphism
$$\Xi_s:\Di(G)\rightarrow\,\Di(G),\qquad \Xi_s(\delta)(f\otimes a)=\delta(f^{p^s}\otimes a).$$

\begin{prop}\label{Frob}
	For each $r,s\in\bN$, the map $\Xi_s$ induces a Hopf algebra homomorphism $\Upsilon_{r,s}:U^{[r]}(G)\rightarrow U^{[r-s]}(G)$.
\end{prop}

\begin{proof}
	First, note that if $f\otimes a\in I_1^{k+1}$, with $f\in\bF_p[G]$ and $a\in\bK$, then $\Xi_s(\delta)(f\otimes a)=\delta(f^{p^s}\otimes a)\subset\delta(I_1^{p^s(k+1)})$. So if $\delta\in\Di_m(G)$ for $m\in\bN$, $\Xi_s(\delta)\in\Di_n(G)$ for $n\geq \frac{m+1}{p^s}-1$. Now, observe that $\delta(1\otimes 1)=0$ implies $\Xi_s(\delta)(1\otimes 1)=0$, so $\delta\in\Di^+_m(G)$ for $m\in\bN$ in fact implies that $\Xi_s(\delta)\in \Di^+_n(G)$ for $n\geq \frac{m+1}{p^s}-1$. We can deduce that if $\delta\in\Di^+_m(G)$ for $m<p^s$ then $\Xi_s(\delta)\in \Di_0^{+}(G)=0$ since $\frac{m+1}{p^s}-1\leq 0$. Hence, $\Xi_s(\Di^+_m(G))=0$ for $m<p^s$. Similarly, if $\delta\in\Di^{+}_{p^{r+1}-1}(G)$ then $\Xi_s(\delta)\in\,\Di_{p^{r-s+1}-1}^+(G)$. 
	
	Furthermore $\Xi_s:\Di^{+}_{p^{r+1}-1}(G)\rightarrow \Di_{p^{r-s+1}-1}^+(G)\hookrightarrow U^{[r-s]}(G)$ is an indexed algebra homomorphism. This follows because if $\delta\in\Di_i^{+}(G)$ and $\mu\in\Di_j^{+}(G)$ with $i+j<p^{r+1}$ then $\Xi_s(\delta)\in \Di_{\tiny \lceil\frac{i+1}{p^s}\rceil-1}^{+}(G)$ and $\Xi_s(\mu)\in \Di_{\tiny \lceil\frac{j+1}{p^s}\rceil-1}^{+}(G)$ (here $\lceil x\rceil$ denotes the smallest integer $\geq x$), and 
	$$\left\lceil\frac{i+1}{p^s}\right\rceil-1 + \left\lceil\frac{j+1}{p^s}\right\rceil-1\leq \frac{i+j}{p^s}<p^{r-s+1},$$ and similarly for the commutator. Hence the universal property gives an algebra homomorphism $\Upsilon_{r,s}:U^{[r]}(G)\rightarrow U^{[r-s]}(G)$.
	
	The fact that $\Upsilon_{r,s}$ is a Hopf algebra homomorphism follows from the fact that $\Xi_s$ is a Hopf algebra homomorphism and the fact that the comultiplication, counit and antipode of $\Upr$ come from the corresponding maps on $\Di(G)$.
	
\end{proof}

Let $M$ be a $G$-module. Since $F_g^s:G\to G$ is an endomorphism, it induces a new $G$-module structure on $M$. We denote $M$ with this induced $G$-module structure by $M^{[s]}$.

\begin{lemma}\label{equiv}
	$\Upsilon_{r,s}:U^{[r]}(G)\to U^{[r-s]}(G)^{[s]}$ is $G$-equivariant for all $r,s\in\bN$.
\end{lemma}

\begin{proof}
	This will follow immediately from the same fact for $\Di^{+}_{p^{r+1}-1}(G)\rightarrow \Di_{p^{r-s+1}-1}^+(G)^{[s]}$. For this to hold, it is enough that the geometric Frobenius commutes with conjugation (where in the codomain the conjugation is pre-composed with the geometric Frobenius). This condition holds since $F_g^s$ is a homomorphism.
\end{proof}

\begin{cor}\label{surj}
	$\Upsilon_{r,s}$ is surjective if $r\geq s$.
\end{cor}

\begin{proof}
	We can see by explicit calculation (cf. \cite{Jan}) that $\Xi_s(\ve_\alpha^{(p^r)})=\ve_\alpha^{(p^{r-s})}$ and $\Xi_s(\binom{\vh_t}{p^r})=\binom{\vh_t}{p^{r-s}}$ for $\alpha\in\Phi$, $1\leq t\leq d$.
\end{proof}


A special case of the previous observation is that when $r=s$ the above process gives a surjective algebra homomorphism $\Upsilon_{r,r}\colon U^{[r]}(G)\to U(\fg)$, and a surjective $G$-module homomorphism $\Upsilon_{r,r}\colon U^{[r]}(G)\to U(\fg)^{[r]}$.


\section{Representation Theory}
\label{sec6}
\subsection{Deformation Algebras}
\label{sec6.1}
In this section we start to consider the representation theory of the algebra $U^{[r]}(G)$. From Theorem \ref{bij}, we have the immediate result:

\begin{cor}
	There is a bijection between the set of (isomorphism classes of) indexed $\Di_{p^{r+1}-1}^+(G)$-modules and the set of (isomorphism classes of) $U^{[r]}(G)$-modules.
\end{cor}

One of the most important differences between the representation theory of Lie algebras in characteristic zero and in positive characteristic is the fact that in characteristic $p>0$ all irreducible representations of $U(\fg)$ are finite-dimensional. Theorem \ref{findim} tells us that we can conclude a similar result for irreducible $U^{[r]}(G)$-modules. The natural question to ask is: how much of the representation theory of $U(\fg)$ can be similarly extended to develop the representation theory of $U^{[r]}(G)$? To that end, let us follow the path well-trodden in the $r=0$ case and see how many difficulties we discover in the generalisation. 

Suppose that $E$ is an irreducible $U^{[r]}(G)$-module, with $G$ reductive. It is finite-dimensional by Theorem \ref{findim}. Hence, by Schur's lemma, $\xi_r(\delta)\in Z_r(G)$ acts as a scalar on $E$ for each $\delta\in\Dipri$. By the semilinearity of $\xi_r$, we can deduce that there exists $\chi_E\in\Dipri^{*}$ (the vector space dual) such that 
$$\xi_r(\delta)\vert_E=\chi_E(\delta)^p\Id_E\qquad \mbox{for all} \,\, \delta\in\Dipri.$$
Note that $\chi_E(\delta)=0\iff \chi_E(\delta)^p=0\iff \xi_r(\delta)=0$. In particular, this means that $\chi_E(X_{p^r})=0$, where $X_{p^r}$ is defined as in \ref{sec5}. 

Recall from Proposition~\ref{Frob} and Corollary~\ref{surj} that $\Upsilon_{r,r}:U^{[r]}(G)\to U(\fg)$ is a surjective algebra homomorphism such that $\Upsilon_{r,r}(\Dipri)=\fg$. The linear map (in fact indexed algebra subspace homomorphism) $\Upsilon_{r,r}\vert_{\tiny\Dipri}\colon\Dipri\to\fg$ has kernel $X_{p^r}$ and hence $\chi_E$ passes to a linear map $\hat{\chi_E}\colon \fg\to \bK$. Similarly, given $\hat{\chi}\in\fg^{*}$ we can extend along $\Upsilon_{r,r}\vert_{\tiny\Dipri}$ to get a linear form $\chi:\Dipri\to\bK$. We shall abuse notation slightly in the following way: given $\chi\in\fg^{*}$, we shall also denote by $\chi$ the linear form $\Dipri\to\bK$ induced by $\Upsilon_{r,r}$.

This allows us to make the following definition for $\chi\in\fg^{*}$:
$$U^{[r]}_\chi(G)=\frac{U^{[r]}(G)}{\langle \xi_r(\delta)-\chi(\delta)^p\,\vert\,\delta\in\Dipri\rangle}.$$
We immediately get the following result:

\begin{prop}\label{irred}
	Every irreducible $U^{[r]}(G)$-module is a $U^{[r]}_\chi(G)$-module for some $\chi\in\fg^{*}$.
\end{prop}

It is straightforward to show that as a vector space over $\bK$ this algebra has dimension $p^{(r+1)\dim(\fg)}$ with basis the classes of 
$$\{\prod_{\alpha\in\Phi^+}\ve_{\alpha}^{\llbracket i_\alpha\rrbracket}\prod_{t=1}^d\binom{\vh_{t}}{\llbracket k_t\rrbracket}\prod_{\alpha\in\Phi^+}\ve_{-\alpha}^{\llbracket j_\alpha\rrbracket }\quad ;\quad 0\leq i_\alpha,j_\alpha,k_t<p^{r+1}\}$$
in $U^{[r]}_\chi(G)$. At times, it will also be beneficial to consider another basis of this algebra, which can be derived easily from properties of divided powers. This basis consists of the classes of
$$\{\prod_{\alpha\in\Phi^+}\ve_{\alpha}^{(i_\alpha)}\prod_{t=1}^d\binom{\vh_{t}}{k_t}\prod_{\alpha\in\Phi^+}\ve_{-\alpha}^{(j_\alpha) }\quad ;\quad 0\leq i_\alpha,j_\alpha,k_t<p^{r+1}\}$$
in $U_\chi^{[r]}(G)$.

Using this basis, and the fact that in $\Di(G_{(r+1)})$ we have $(\ve_\alpha^{(p^r)})^{p}=0$ and $\binom{\vh_t}{p^r}^p=\binom{\vh_t}{p^r}$, it is straightforward to show that $U_0^{[r]}(G)=\Di(G_{(r+1)})$. One can also show that, for $\chi\in\fg^{*}$ and $s\leq r$, we get that $\Upsilon_{r,r-s}\colon U^{[r]}_\chi(G)\to U^{[s]}_\chi(G)$ is a well-defined algebra homomorphism. So we get the sequence of algebra homomorphisms
$$U_\chi^{[r]}(G)\twoheadrightarrow U^{[r-1]}_\chi(G)\twoheadrightarrow\ldots\twoheadrightarrow U^{[1]}_\chi(G)\twoheadrightarrow U_\chi(\fg).$$

Given $g\in G$ we get an adjoint action of $g$, $\Ad(g)$, on $\Dipri$. This leads to a coadjoint action of $g$ on $\Dipri^{*}$. We furthermore have a twisted coadjoint action of $g$ on $(\fg^{*})^{[r]}$, corresponding to the twisted adjoint action $\Ad(F_g^r(g))$. 

\begin{lemma}\label{orbit}
	Given $(\chi\in\fg^{*})^{[r]}$ and $g\in G$, there is an isomorphism $U^{[r]}_\chi(G)\cong U^{[r]}_{g\cdot\chi}(G)$.
\end{lemma}

\begin{proof}
	
	Consider the coadjoint actions of $G$ on $\Dipri^{*}$ and on $\fg^{*}$ (untwisted and twisted respectively). A priori, the actions need not be compatible when we switch between considering $(\chi\in \fg^{*})^{[r]}$ as a linear form on $\fg$ and a linear form on $\Dipri$. However, the $G$-equivariance of $\Upsilon_{r,r}$ (see Lemma~\ref{equiv}) means that this is not a problem - the actions are compatible.
	
	As a result, one can show by inspection that  
	$$U^{[r]}_\chi(G)\cong U^{[r]}_{g\cdot\chi}(G)$$
	where we mean by $g\cdot\chi$ the twisted coadjoint action of $g$ on $\chi$ - by Section~\ref{sec4.1}, it doesn't matter here if we consider the action of $g$ on $(\chi\in \fg^{*})^{[r]}$ or $\chi\in \Dipri^{*}$.
\end{proof}
In particular, much like in the $r=0$ case, to understand the representation theory of $U^{[r]}(G)$ it is enough to understand the representation theory of $U_\chi^{[r]}(G)$ for $(\chi\in \fg^{*})^{[r]}$ in distinct $G$-orbits.

\subsection{Frobenius Kernels}
\label{sec6.2}
We would now like to show that $\Di(G_{(r)})$ is a subalgebra of $\Uprc$ for any choice of $\chi\in\fg^{*}$. We saw earlier that 
$$\Di(G_{(r+1)})\cong\frac{\Upr}{\langle \delta^{\otimes p}-\delta^p\,\vert\,\delta\in\Dipri\rangle},$$
so it is enough to show that 
$$\Di(G_{(r)})\cong\frac{U^{[r-1]}(G)}{\langle \delta^{\otimes p}-\delta^p\,\vert\,\delta\in\Di^+_{p^{r-1}}(G)\rangle}\hookrightarrow \frac{\Upr}{\langle \delta^{\otimes p}-\delta^p-\chi(\delta)^p1\,\vert\,\delta\in\Dipri\rangle}.$$

Inclusion gives us a map $i\colon \Di^{+}_{p^{r}-1}(G)\hookrightarrow \Di^{+}_{p^{r+1}-1}(G)\hookrightarrow\Upr$ which clearly satisfies all the conditions for the universal property, so we get an algebra homomorphism 
$$\overline{i}\colon U^{[r-1]}(G)\rightarrow \Upr\twoheadrightarrow \Uprc.$$
It is straightforward to see from the basis description of $\Upr$ that $\Img(\overline{i})\cap \langle \delta^{\otimes p}-\delta^p-\chi(\delta)^p1\,\vert\,\delta\in\Dipri\rangle=0$, so we just need to show that $\ker(\overline{i})=\langle \delta^{\otimes p}-\delta^p\,\vert\,\delta\in\Di^+_{p^{r-1}}(G)\rangle$. This follows easily from the basis descriptions of $U^{[r-1]}(G)$ and $U^{[r]}(G)$ once we notice that $\overline{i}(\ve_\alpha^{(p^{r-1})^p})=0$ and  $\overline{i}(\binom{\vh_t}{p^{r-1}}^p)=\binom{\vh_t}{p^{r-1}}$.

In particular, we have the following diagram of inclusions and projections:
$$
\xymatrix{
	\ldots \ar@{->>}[drr]^{} & & U^{[r-1]}(G) \ar@{->>}[drr]^{} \ar@{<-^{)}}[d]^{} & & U^{[r]}(G) \ar@{->>}[drr]^{} \ar@{<-^{)}}[d]^{} & & U^{[r+1]}(G) \ar@{->>}[drr]^{} \ar@{<-^{)}}[d]^{}\\
	\ldots\ar@{^{(}->}[rr]^{} & & \Di(G_{(r-1)}) \ar@{^{(}->}[rr]^{} & & \Di(G_{(r)}) \ar@{^{(}->}[rr]^{} & & \Di(G_{(r+1)})\ar@{^{(}->}[rr]^{} & & \ldots\\
}
$$
This hence provides us with a direct system $\ldots\to U^{[r-1]}(G)\to U^{[r]}(G)\to U^{[r+1]}(G)\to\ldots$ with direct limit $\varinjlim U^{[r]}(G)=\Di(G)$.
From what we have already shown, we can use this to deduce some details of the module theory of $U_\chi^{[r]}(G)$. 

\begin{prop}\label{DistGr}
	Every $U^{[r]}_\chi(G)$-module is a $\Di(G_{(s)})$-module for all $0\leq s\leq r$.
\end{prop}

\begin{prop}
	Every $U^{[s]}_\chi(G)$-module can be lifted to a $U_\chi^{[r]}(G)$-module via $\Upsilon_{r,r-s}$.
\end{prop}

We can put these two results together in the following theorem. The proof follows easily from Section~\ref{sec4.1}.

\begin{prop}
	Let $M$ be a $U_\chi^{[r]}(G)$-module. If $M$ is lifted from a $U^{[s]}_\chi(G)$-module along $\Upsilon_{r,r-s}$ then $\Di^+(G_{(s)})M=0$. On the other hand, if $\Di^+(G_{(s)})M=0$, then $M$ is a $U_\chi^{[s]}(G)$-module via a lifting along $\Upsilon_{r,s}$.
\end{prop}







\subsection{Examples}
\label{sec6.3}
\begin{exa}
	Consider the additive algebraic group $G=\bG_a$. We know from \cite[I.7.8]{Jan} that $\Di_{p^{r+1}-1}(G)$ has basis $\gamma_1,\gamma_2,\dots,\gamma_{p^{r+1}-1}$ and that in $\Di(G)$ the multiplication is $\gamma_k\gamma_l=\binom{k+l}{k}\gamma_{k+l}$. Using these facts one can show that
	$$U^{[r]}(\bG_a)=\frac{\bK[t_0,t_1,\ldots,t_r]}{\langle t_i^p\,\vert\,0\leq i\leq {r-1}\,\rangle}.$$
	Furthermore, given $\chi\in\fg^{*}=\bK$, we get 
	$$U^{[r]}_\chi(\bG_a)=\frac{\bK[t_0,t_1,\ldots,t_r]}{\langle t_{r}^p-\chi^p ; t_i^p\,\vert\,0\leq i\leq {r-1}\,\rangle}\cong \frac{\bK[t]}{\langle t^p\rangle}\times \ldots\times \frac{\bK[t]}{\langle t^p\rangle}\times \frac{\bK[t]}{\langle t^p-\chi^p\rangle}.$$
\end{exa}

\begin{exa}\label{multi}
	Consider the multiplicative algebraic group $G=\bG_m$. We know from \cite[I.7.8]{Jan} that $\Di_{p^{r+1}-1}(G)$ has basis $\delta_1,\delta_2,\dots,\delta_{p^{r+1}-1}$ and that in $\Di(G)$ the multiplication is $\delta_k\delta_l=\sum_{i=0}^{\min(k,l)}\frac{(k+l-i)!}{(k-i)!(l-i)!i!}\delta_{k+l-i}$. Using these facts one can show that
	$$U^{[r]}(\bG_m)=\frac{\bK[t_0,t_1,\ldots,t_r]}{\langle t_i^p-t_i\,\vert\,0\leq i\leq {r-1}\,\rangle}.$$
	Furthermore, given $\chi\in\fg^{*}=\bK$, we get 
	$$U^{[r]}_\chi(\bG_m)=\frac{\bK[t_0,t_1,\ldots,t_r]}{\langle t_{r}^p-t_r-\chi^p ; t_i^p-t_i\,\vert\,0\leq i\leq {r-1}\,\rangle}\cong \bK\times\ldots\times\bK$$
	where there are $rp$ copies of $\bK$ in the final expression, since $t_i^p-t_i$ and $t_r^p-t_r-\chi^p$ are separable polynomials. This tells us that the algebra $U^{[r]}_\chi(\bG_m)$ is semisimple.
\end{exa}

\subsection{Higher Baby Verma Modules}
\label{sec6.4}
One of the main constructions which we use to study $U_\chi(\fg)$-modules is that of baby Verma modules. We would like to construct a similar module for this higher case. We shall assume that there exists a $G$-invariant non-degenerate bilinear form on $\fg$, so that as in \cite{Jan2} we can also assume that $\chi(\fn^+)=0$. One of the benefits of the work we have done so far is that assumptions like these are nothing new - any conditions on $\chi$ hold independently of the power of $p$ we are working with, except that the $G$-action is twisted by the corresponding geometric Frobenius. In particular, this is a reasonable assumption for exactly the same reasons as in the standard case. This assumption in this case also tells us that, when $\chi$ is viewed as a linear form on $\Dipri$, we have that $\chi(\Di_{p^r}^+(U))=0$.

Let $0\neq M$ be a $U_\chi^{[r]}(B)$-module. We have that $\chi(\ve_\alpha^{(k)})=0$ for all $\alpha\in\Phi$ and $0<k\leq p^r$ and that $\xi_r(\ve_\alpha^{(k)})=(\ve_\alpha^{(k)})^{\otimes p}$. This means that every $\ve_\alpha^{(k)}$ acts nilpotently on $M$. As a result, using an argument similar to that of Rudakov in \cite{Rud} we get that
$$\{m\in M\,\vert\,\Di_{p^{r+1}-1}^+(U)m=0\}\neq 0.$$

Let us consider the action of $\Di_{p^{r+1}-1}^+(T)$ on this set. Since $\Di_{p^{r+1}-1}^+(T)$ is commutative as an indexed algebra subspace (i.e. $\delta\mu=\mu\delta$ whenever $\delta\in \Di_{i}^+(T)$ and $\mu\in\Di_{j}^+(T)$ with $i+j<p^{r+1}$) and $U_\chi^{[r]}(T)$ is semi-simple by Example \ref{multi} there exists $0\neq m_0\in M$ such that $\Di_{p^{r+1}-1}^+(U)m=0$ and there exists $\lambda\in\Di_{p^{r+1}-1}^+(T)^{\overline{*}}$ (the indexed algebra subspace dual) such that, for each $\delta\in \Di_{p^{r+1}-1}^+(T)$, $\delta m_0=\lambda(\delta)m_0$.

Now, given $\lambda\in \Di_{p^{r+1}-1}^+(T)^{\overline{*}}$, we can define the one-dimensional $\Di_{p^{r+1}-1}^+(B)$-module $\bK_\lambda$ where $\Di_{p^{r+1}-1}^+(U)$ acts as zero and $\delta\in\Di_{p^{r+1}-1}^+(T)$ acts as multiplication by $\lambda(\delta)$. This $\Di_{p^{r+1}-1}^+(B)$-module will give a $U^{[r]}_\chi(B)$-module if and only if $\lambda\in\Lambda^r_\chi$ where 
$$\Lambda^r_\chi\coloneqq\{\lambda \in\Di_{p^{r+1}-1}^+(T)^{\overline{*}}\,\vert\,\lambda(\delta)^p-\lambda(\delta^p)=\chi(\delta)^p\,\mbox{for all}\,\delta\in \Di_{p^r}^+(T)\}.$$
Note that for $\delta\in \Di_{p^r}^+(T)$ we in fact have that $\delta^p=\delta$, so the criterion can also be written as $\lambda(\delta)^p-\lambda(\delta)=\chi(\delta)^p$ for all $\delta\in \Di_{p^r}^+(T)$. A necessary and sufficient condition for $\lambda\in \Di_{p^{r+1}-1}^+(T)^{\overline{*}}$ to lie inside $\Lambda^r_\chi$ is hence that $\lambda(\binom{\vh_t}{p^k})^p-\lambda(\binom{\vh_t}{p^k})=0$ for $1\leq k<r$ and $\lambda(\binom{\vh_t}{p^r})^p-\lambda(\binom{\vh_t}{p^r})=\chi(\binom{\vh_t}{p^r})^p$, for all $1\leq t\leq d$.

Given $\lambda\in\Lambda^r_\chi$ we can hence define the \emph{higher baby Verma module}:
$$Z_\chi^r(\lambda)=U^{[r]}_\chi(G)\otimes_{U^{[r]}_\chi(B)} \bK_\lambda.$$
Letting $v_\lambda=1\otimes 1$, we get that a basis of $Z^r_\chi(\lambda)$ is 
$$\{\prod_{\alpha\in\Phi^{+}}\ve_{-\alpha}^{(i_\alpha)}v_\lambda\quad :\quad 0\leq i_\alpha<p^{r+1}\}$$
and thus that $Z^r_\chi(\lambda)$ has dimension $p^{(r+1)\vert\Phi^{+}\vert}$.

As in the $r=0$ case, Frobenius reciprocity gives the following lemma:

\begin{lemma}\label{Verma}
	Every irreducible $U_\chi^{[r]}(G)$-module is a homomorphic image of $Z^r_\chi(\lambda)$ for some $\lambda\in \Lambda^r_\chi$.
\end{lemma}

Observe that when $r=0$, we just get baby Verma modules as in the existing theory.

\section{Special Linear Group}
\label{sec7}
Let us examine the particular case of the algebraic group $G=SL_2$, and try to understand the module theory of $U_\chi^{[r]}(SL_2)$ for $(\chi\in\fs\fl_2^{*})^{[r]}$. We assume $p>2$ in this section. Recall from Lemma~\ref{orbit} that to understand the irreducible modules of $U^{[r]}(SL_2)$ it is enough to understand the irreducible $U^{[r]}_\chi(SL_2)$-modules up to the $G$-orbit of $\chi$ under the twisted coadjoint action.  

We observe that the $G$-orbits of $(\fs\fl_2^{*})^{[r]}$ are the same as the $G$-orbits of $\fs\fl_2^{*}$. This follows from Proposition I.9.5 in \cite{Jan}, since $SL_2^{(r)}$ and $SL_2$ are isomorphic through the arithmetic Frobenius homomorphism. It is well-known (see, for example, Section 5.4 in \cite{Jan2}) that each element of $\fs\fl_2^{*}$ is conjugate under the $SL_2$-action to a linear form of one of the following types:
$$\ve\mapsto 0\,\qquad\quad \vf\mapsto 0\,\qquad\quad  \vh\mapsto t-s,$$
$$\ve\mapsto 0\,\qquad\quad  \vf\mapsto 1\,\qquad\quad \vh\mapsto 0,$$
where $t,s\in\bK$ and we are using the standard notation of $\ve,\vh,\vf\in\fs\fl_2$ to mean

$$\ve=\begin{pmatrix}
0 & 1\\
0 & 0\\
\end{pmatrix},
\quad \vh=\begin{pmatrix}
1 & 0\\
0 & -1\\
\end{pmatrix},
\quad \vf=\begin{pmatrix}
0 & 0\\
1 & 0\\
\end{pmatrix}.
$$

A linear form conjugate to the first type is called {\em semisimple}, and a linear form conjugate to the second type (or 0) is called {\em nilpotent}. From now on we shall assume that $\chi$ takes one of the above forms. In the rest of this chapter we shall classify the irreducible $\Uprc$-modules for $\chi$ non-zero semisimple, $\chi$ non-zero nilpotent, and $\chi=0$ in Subsections~\ref{sec7.2}, \ref{sec7.3} and \ref{sec7.4} respectively.

From the above discussion, we see that given $\chi\in\fs\fl_2^{*}$ and $\lambda\in\Lambda^r_\chi$ we can form the higher baby Verma module $Z^r_\chi(\lambda)$. In the case of $SL_2$, this module has basis $\{v_i\,\vert\, 0\leq i<p^{r+1}\}$, where we denote $v_i\coloneqq f^{(i)}\otimes m_0$ - here $m_0$ is a generator of $\bK_\lambda$. With a little work, we can write down how generators of $U_\chi^{[r]}(G)$ act on the basis:

$$\ve^{(p^j)}v_i=\twopartdef{0}{p^j>i,}{(\sum_{t=0}^{p^j}\lambda(\binom{\vh}{t})\binom{p^j-i}{p^j-t})v_{i-p^j}}{p^j\leq i,}$$

$$\binom{\vh}{p^j}v_i=(\sum_{t=0}^{p^j}\binom{-2i}{p^j-t}\lambda(\binom{\vh}{t}))v_i,$$

\begin{equation}\label{fact2}
	\vf^{(p^j)}v_i=\threepartdef{\binom{p^j+i}{p^j}v_{i+p^j}}{i+p^j<p^{r+1},}{0}{j\neq r\,\mbox{and}\,i+p^j\geq p^{r+1},}{\frac{1}{(p-1)!}\chi(\vf^{(p^r)})^pv_k}{j=r\,\mbox{and}\, i+p^r=p^{r+1}+k\,\, \mbox{for}\,\, k\geq 0.}
\end{equation}
Note that here we are defining $\lambda(1)$ to be equal to 1.

In fact, we can even say that 
\begin{equation}\label{eact2}
	\ve^{(l)}v_k=\twopartdef{0}{l>k,}{\lambda(\binom{\vh -k+l}{l})v_{k-l}=(\sum_{t=0}^{l}\lambda(\binom{\vh}{t})\binom{l-k}{l-t})v_{k-l}}{l\leq k}
\end{equation}
and 
\begin{equation}\label{hact2}
	\binom{\vh}{l}v_k=\lambda(\binom{\vh-2k}{l})v_k=(\sum_{t=0}^{l}\binom{-2k}{l-t}\lambda(\binom{\vh}{t}))v_k
\end{equation}
for $0\leq k,l<p^{r+1}$.

Before going any further, let us recall some properties of the divided powers of $\ve,\vh,\vf$. Suppose that $k=a_0+a_1p+\ldots + a_{r-1}p^{r-1}+a_rp^r$ with $0\leq a_i<p$ for each $i$. Then we have
$$\ve^{(k)}=\ve^{(a_0)}\ve^{(a_1p)}\ldots\ve^{(a_{r-1}p^{r-1})}\ve^{(a_rp^r)}=\frac{1}{a_0!a_1!\ldots a_r!}\ve^{a_0}(\ve^{(p)})^{a_1}\ldots(\ve^{(p^{r-1})})^{a_{r-1}}(\ve^{(p^r)})^{a_r},$$

$$\vf^{(k)}=\vf^{(a_0)}\vf^{(a_1p)}\ldots\vf^{(a_{r-1}p^{r-1})}\vf^{(a_rp^r)}=\frac{1}{a_0!a_1!\ldots a_r!}\vf^{a_0}(\vf^{(p)})^{a_1}\ldots(\vf^{(p^{r-1})})^{a_{r-1}}(\vf^{(p^r)})^{a_r},$$
\begin{multline*}
	\binom{\vh}{k}=\binom{\vh}{a_0}\binom{\vh}{a_1p}\ldots\binom{\vh}{a_{r-1}p^{r-1}}\binom{\vh}{a_rp^r}=\\=\frac{1}{a_0!a_1!\ldots a_r!}\vh(\vh-1)\ldots(\vh-a_0+1)(\binom{\vh}{p})(\binom{\vh}{p}-1)\ldots (\binom{\vh}{p^r}-a_r+2)(\binom{\vh}{p^r}-a_r+1).
\end{multline*}
Since $\lambda$ is an indexed algebra subspace homomorphism, this tells us that 
\begin{multline*}
	\lambda(\binom{\vh}{k})=\frac{1}{a_0!a_1!\ldots a_r!}\lambda(\vh)(\lambda(\vh)-1)\ldots(\lambda(\vh)-a_0+1)(\lambda(\binom{\vh}{p}))\times\\\times(\lambda(\binom{\vh}{p})-1)\ldots (\lambda(\binom{\vh}{p^r})-a_r+2)(\lambda(\binom{\vh}{p^r})-a_r+1).
\end{multline*}
Another useful observation to make is that $\ve^{(t)}v_t=\lambda(\binom{\vh}{t})v_0$.
Let us now examine the different cases for $\chi$.

\subsection{Non-zero semisimple $\chi$}
\label{sec7.2}

The definition of $\Lambda^r_\chi$ in this case tells us that $\lambda(\binom{\vh}{p^r})^p-\lambda(\binom{\vh}{p^r})=\chi(\binom{\vh}{p^r})^p\neq 0$ and hence that $\lambda(\binom{\vh}{p^r})\notin \bF_p$. We also know that $\chi(\vf^{(p^r)})=0$. In particular, using the above notation, we get that $\lambda(\binom{\vh}{k})=0$ if and only if $a_i>\lambda(\binom{\vh}{p^i})$ for some $0\leq i<r$ (here we are abusing notation slightly to treat $\bF_p$ as the subset $\{0,1,2,\ldots, p-1\}$ of the integers - observing that $\lambda(\binom{\vh}{p^i})^p=\lambda(\binom{\vh}{p^i})$ for 
$0\leq i<r$ means that all such $\lambda(\binom{\vh}{p^i})$ lie inside $\bF_p$).

Consider the vector subspace of $Z^r_\chi(\lambda)$ with basis
$$\{v_k\,\vert\,\lambda(\binom{\vh}{k})=0\}.$$ 
By above this is the same as 
$$\{v_{a_0+a_1p+\ldots+ a_{r-1}p^{r-1}+a_rp^r}\,\vert\, a_i>\lambda(\binom{\vh}{p^i})\, \mbox{for some}\, i\in\{0,1,\ldots,r-1\}\}.$$
We shall denote this subspace by $M^r_\chi(\lambda)$. Note that $M^0_\chi(\lambda)=0$ for all choices of $\chi,\lambda$. 

\begin{lemma}\label{sssub}
	$M^r_\chi(\lambda)$ is a $U_\chi^{[r]}(G)$-submodule of $Z^r_\chi(\lambda)$.
\end{lemma}

\begin{proof}
	We need to show that this subspace is preserved by $\ve^{(l)}$, $\binom{\vh}{l}$ and $\vf^{(l)}$ for $0\leq l<p^{r+1}$. It is clearly preserved by all $\binom{\vh}{l}$, so we just need to show it for $\ve^{(l)}$ and  $\vf^{(l)}$. 
	
	Let $l=a_0+a_1p+\ldots a_{r-1}p^{r-1}+a_rp^r$ and $k=b_0+b_1p+\ldots+b_{r-1}p^r$ with $0\leq a_i,b_i<p$ for all $i$. Then we have
	$$\vf^{(l)}v_k=\frac{1}{a_0!a_1!\ldots a_{r-1}!a_r!b_0!b_1!\ldots b_{r-1}!b_r!}\vf^{a_0+b_0}(\vf^{(p)})^{a_1+b_1}\ldots(\vf^{(p^{r-1})})^{a_{r-1}+b_{r-1}}(\vf^{(p^r)})^{a_r+b_r}\otimes m_0.$$	
	If $a_i+b_i\geq p$ for some $0\leq i<r$ then this expression is zero, since $(\vf^{(p^i)})^{p}=0$. If $a_i+b_i<p$ for all $0\leq i\leq r$, then we have just increased the exponent in each term, which clearly will preserve $M^r_\chi(\lambda)$.
	
	The only remaining case is if $a_i+b_i<p$ for all $0\leq i< r$ and $a_r+b_r= p+s$ for some $0\leq s<p$. In this case, we get that
	$$\vf^{(l)}v_k=\frac{\chi(\vf^{(p^r)})^p}{a_0!a_1!\ldots a_{r-1}!a_r!b_0!b_1!\ldots b_{r-1}!b_r!}\vf^{a_0+b_0}(\vf^{(p)})^{a_1+b_1}\ldots(\vf^{(p^{r-1})})^{a_{r-1}+b_{r-1}}(\vf^{(p^r)})^{s}\otimes m_0=0$$ 
	as $\chi(\vf^{(p^r)})=0$ by assumption.	Hence, we get that $M^r_\chi(\lambda)$ is preserved by the $\vf^{(l)}$. 
	
	Observe that over $\bC$, we have for $l\leq k<p^{r+1}$ that $\sum_{t=0}^{l}\binom{\vh}{t}\binom{l-k}{l-t}=\binom{\vh+l-k}{l}$ and that $\binom{\vh+l-k}{l}\binom{\vh}{k-l}=\binom{k}{l}\binom{\vh}{k}$. In particular, this means that in $U_\chi^{[r]}(G)$ we have $$	(\sum_{t=0}^{l}\binom{\vh}{t}\binom{l-k}{l-t})\binom{\vh}{k-l}=\binom{k}{l}\binom{\vh}{k}$$
	and hence
	\begin{equation}\label{mult}
		(\sum_{t=0}^{l}\lambda(\binom{\vh}{t})\binom{l-k}{l-t})\lambda(\binom{\vh}{k-l})=\binom{k}{l}\lambda(\binom{\vh}{k}).
	\end{equation}
	
	Now let us compute $\ve^{(l)}v_k$ for $k$ with $\lambda(\binom{\vh}{k})=0$. When $l>k$ the expression is $0$ and the result follows, so we may assume $l\leq k$. Using Equation (\ref{eact2}), we get that	
	$$\ve^{(l)}v_k= (\sum_{t=0}^{l}\lambda(\binom{\vh}{t})\binom{l-k}{l-t})v_{k-l}.$$
	Since $\lambda(\binom{\vh}{k})=0$, Equation (\ref{mult}) tells us that either $(\sum_{t=0}^{l}\lambda(\binom{\vh}{t})\binom{l-k}{l-t})=0$, in which case $\ve^{(l)}v_k=0$ by above, or $\lambda(\binom{\vh}{k-l})=0$ in which case $v_{k-l}\in M^r_\chi(\lambda)$ and so $\ve^{(l)}v_k\in M^r_\chi(\lambda)$.
	
	In either case, $M^r_\chi(\lambda)$ is preserved by the $\ve^{(l)}$ and we are done.	
	
\end{proof}

\begin{prop}\label{ssmax}
	$M^r_\chi(\lambda)$ is a maximal submodule of $Z^r_\chi(\lambda)$.
\end{prop}

\begin{proof}
	It is enough to show that the quotient module $L^r_\chi(\lambda)\coloneqq Z^r_\chi(\lambda)/M^r_\chi(\lambda)$ is irreducible. From the above description it is clear that $L^r_\chi(\lambda)$ has as basis the images under the quotient map of the elements
	$$\{v_k\,\vert\,\lambda(\binom{\vh}{k})\neq 0\}.$$
	We shall abuse notation to denote by $v_k$ both the element in $Z^r_\chi(\lambda)$ and its image in the quotient.
	
	We can see that $\ve^{(k)}v_k=\lambda(\binom{\vh}{k})v_0\neq 0$ for the $v_k$ in this basis.
	
	Let $N$ be a non-zero submodule of $L^r_\chi(\lambda)$. There hence exists a non-zero element	$v=\sum \alpha_t v_t\in N$ where the sum is over all $0\leq t<p^{r+1}$ with $\lambda(\binom{\vh}{t})\neq 0$. Suppose $s$ is the largest such element with $\alpha_s\neq 0$. Then we have
	$$\ve^{(s)}v=\sum \alpha_t\ve^{(s)}v_t=\alpha_s\lambda(\binom{\vh}{s})v_0\neq 0.$$
	Hence $v_0\in N$ and it easy to see that we must have $N=L^r_\chi(\lambda)$, so $L^r_\chi(\lambda)$ is irreducible.
\end{proof}


\begin{prop}\label{ssuniq}
	$M^r_\chi(\lambda)$ is the unique maximal submodule of $Z^r_\chi(\lambda)$.
\end{prop}

\begin{proof}
	Let $N$ be a maximal submodule of $Z^r_\chi(\lambda)$ with $N\neq M^r_\chi(\lambda)$. Hence, there exists
	$$w=\sum_{i=0}^{p^{r+1}-1}\alpha_iv_i\in N$$	
	with $\alpha_i\neq 0$ for at least one $i$ with $\lambda(\binom{\vh}{i})\neq 0$. Let $l$ be the largest such integer. Then we have
	$$\ve^{(l)}w=\sum_{i=0}^{p^{r+1}-1}\alpha_i\ve^{(l)}v_i=\alpha_l\lambda(\binom{\vh}{l})v_0 + \sum_{i=1}^{p^{r+1}-1} \beta_iv_i\in N$$
	where we have $\beta_i=0$ whenever $\lambda(\binom{\vh}{i})\neq 0$. This follows from the description of the action and from Equation (\ref{mult}). Let $k_1<\ldots<k_s$ be the integers between 1 and $p^{r+1}-1$ with $\lambda(\binom{\vh}{k_i})=0$. Rescaling, we can assume
	$$v\coloneqq v_0 + \sum_{i=1}^{s}\gamma_i^0v_{k_i}\in N$$ 
	where $\gamma_i^0\in\bK$ (here the superscripts are used for indexing). 
	
	Now, we see that for $1\leq j\leq s$ we have	
	$$ \vf^{(k_j)}v=\vf^{(k_j)}v_0+\sum_{i=1}^{s}\gamma_i^0\vf^{(k_j)}v_{k_i} =v_{k_j} + \sum_{i=1}^{s}\gamma_i^0\binom{k_i+k_j}{k_j}v_{k_i+k_j}\in N.$$
	We know $\vf^{(k_j)}v_{k_i}=\binom{k_i+k_j}{k_j}v_{k_i+k_j}$.
	Since $\chi(\vf^{(p^r)})=0$, we get that $\vf^{(k_j)}v_{k_i}=0$ if $k_i+k_j>p^{r+1}$. On the other hand, $\vf^{(k_j)}v_{k_i}\in M^r_\chi(\lambda)$ for $i=1,\ldots,s$, since $M^r_\chi(\lambda)$ is a submodule of $Z^r_\chi(\lambda)$. Thus, $\vf^{(k_j)}v_{k_i}$ is a linear combination of those basis elements of $M_\chi(\lambda)$ whose index is at least $k_i+k_j$. In other words, 	
	$$\vf^{(k_j)}v= v_{k_j} + \sum_{i=j+1}^{s}\gamma_i^jv_{k_i}\in N$$
	for some $\gamma_i^j\in \bK$.
	
	In particular, this tells us that
	$$\vf^{(k_s)}v= v_{k_s}\in N,$$
	$$\vf^{(k_{s-1})}v=v_{k_{s-1}}+ \gamma_{s}^{s-1}v_{k_s}\in N,$$
	$$\vf^{(k_{s-2})}v=v_{k_{s-2}}+\gamma_{s-1}^{s-2}v_{k_s-1}+ \gamma_{s}^{s-2}v_{k_s}\in N$$
	and so on. This tells us inductively that $v_{k_i}\in N$ for all $1\leq i\leq s$, and hence that $M^r_\chi(\lambda)\subset N$. But since we know that $M^r_\chi(\lambda)\neq N$, we must have that $N=Z^r_\chi(\lambda)$, contradicting maximality.
	
	Hence, $M^r_\chi(\lambda)$ is the unique maximal submodule of $Z^r_\chi(\lambda)$.
\end{proof}

We have constructed all irreducible $U_\chi^{[r]}(SL_2)$-modules in the case when $\chi\neq 0$ is semisimple. 
Given $\lambda\in\Lambda^r_\chi$, we therefore get a unique irreducible $U_\chi^{[r]}(SL_2)$-module of dimension:
$$(\lambda(\vh)+1)(\lambda(\binom{\vh}{p})+1)\ldots (\lambda(\binom{\vh}{p^{r-1}})+1)p,$$
where we view $\lambda(\binom{\vh}{p^i})$ as elements of the set $\{0,1,2,\ldots,p-1\}$ for $0\leq i<r$.

We can also say something about the structure of these irreducible modules as $\Di(G_{(r)})$ modules.

\begin{prop}\label{DistDec}
	Each irreducible $U_{\chi}^{[r]}(G)$-module $M$ decomposes as $\Di(G_{(r)})$-modules into a direct sum of $p$ copies of the same irreducible $\Di(G_{(r)})$-module. In particular, if $M$ is a quotient of $Z_\chi^r(\lambda)$ for $\lambda\in \Lambda_\chi^r$ then this irreducible $\Di(G_{(r)})$-module is the module $L_r(\lambda)$ defined in \cite{Jan}.
\end{prop}

Note that we are abusing notation slightly to identify $\lambda\in \Lambda_\chi^r$ with $\lambda\in X_r(T)$ (as defined in \cite{Jan}), but it is straightforward to see how this can be done. Specifically, each $\lambda\in\Di_{p^{r+1}-1}^+(T)^{\overline{*}}$ with
$\lambda(\binom{\vh}{p^k})^p-\lambda(\binom{\vh}{p^k})=0$ for $1\leq k<r$ and $\lambda(\binom{\vh}{p^r})^p-\lambda(\binom{\vh}{p^r})=\chi(\binom{\vh}{p^r})^p$, restricts to $\lambda\in\Di_{p^{r}-1}^+(T)^{\overline{*}}$ with $\lambda(\binom{\vh}{p^k})^p-\lambda(\binom{\vh}{p^k})=0$ for $1\leq k<r$. This then induces an algebra homomorphism $\lambda:\Di(T_{(r)})\to \bK$, which gives $\lambda\in X_r(T)$. This argument is independent of $\chi$, so shall also be used in subsequent sections.

\begin{proof}
	By Lemma~\ref{sssub} and Propositions~\ref{ssmax} and \ref{ssuniq}, we know that each irreducible $U_\chi^{[r]}(G)$-module is obtained as the unique irreducible quotient module of $Z_\chi^r(\lambda)$ for some $\lambda\in\Lambda_\chi^r$ and has as basis the images under the quotient map of
	$$\{v_k\,\vert\,\lambda(\binom{\vh}{k})\neq 0\,\mbox{and}\,\,0\leq k<p^{r+1}\}$$
	or equivalently those of
	$$\{v_{a_0+a_1p+\ldots+ a_{r-1}p^{r-1}+a_rp^r}\,\vert\, 0\leq a_i\leq\lambda(\binom{\vh}{p^i})\, \mbox{for all}\, i\in\{0,1,\ldots,r-1\},\,\, 0\leq a_r< p\}.$$
	We shall denote this module as $L_\chi^r(\lambda)$. We shall abuse notation to denote by $v_k$ both the element in $Z^r_\chi(\lambda)$ and its image in $L_\chi^r(\lambda)$.
	
	Let $N$ be the subspace of $L_\chi^r(\lambda)$ with basis 
	$$\{v_k\,\vert\,\lambda(\binom{\vh}{k})\neq 0\,\mbox{and}\,\,0\leq k<p^r\}=\{v_{a_0+a_1p+\ldots+ a_{r-1}p^{r-1}}\,\vert\, a_i\leq\lambda(\binom{\vh}{p^i})\, \mbox{for all}\, i\in\{0,1,\ldots,r-1\}\}.$$
	We claim that $N$ is a $\Di(G_{(r)})$-module. It suffices, by Equations \ref{eact2} and \ref{hact2}, to check that $\vf^{(l)}N\subset N$ for $0\leq l<p^r$. This, however, is clear since $\vf^{(l)}v_k=\binom{k+l}{l}v_{k+l}$ and so either $k+l<p^r$ and we are done, or $k+l\geq p^r$ and then Lucas' Theorem gives $\binom{k+l}{l}=0$.
	
	From the basis description and Lucas' Theorem it is straightforward to see that $L_\chi^r(\lambda)=\bigoplus_{a=0}^{p-1}\vf^{(ap^r)}N$. So to get our result, we just need to show that $N$ is irreducible and that $N\cong \vf^{(ap^r)}N$ as $\Di(G_{(r)})$-modules for all $0\leq a<p$. That $N$ is irreducible (and in fact isomorphic to $L_r(\lambda)$) follows easily from the well known representation theory of $\Di(G_{(r)})$ (see \cite[II.3]{Jan} or Subsection~\ref{sec7.4} below).
	
	We now define, for $0\leq a<p$, the linear map $\phi_a:N\to \vf^{(ap^r)}N$ which is defined on the basis of $N$ as 
	$$v_{a_0+a_1p+\ldots+ a_{r-1}p^{r-1}}\mapsto \vf^{(ap^r)}v_{a_0+a_1p+\ldots+ a_{r-1}p^{r-1}}=v_{a_0+a_1p+\ldots+ a_{r-1}p^{r-1} +ap^r}.$$
	This map is clearly surjective and, by the irreducibility of $N$, it will be injective once we show that $\phi_a$ is a homomorphism of $\Di(G_{(r)})$-modules. To show this, it suffices to show for $0\leq l,z<p^r$ that $\vf^{(l)}\phi_a(v_z)=\phi_a(\vf^{(l)}v_z)$, $\binom{\vh}{l}\phi_a(v_z)=\phi_a(\binom{\vh}{l}v_z)$ and $\ve^{(l)}\phi_a(v_z)=\phi_a(\ve^{(l)}v_z)$.
	
	
	That $\vf^{(l)}\phi_a(v_z)=\phi_a(\vf^{(l)}v_z)$ follows easily from the action of $U_\chi^{[r]}(G)$ on $Z_\chi^r(\lambda)$ (Equation (\ref{fact2})) and Lucas' Theorem. 
	
	Observe that if $b_0,b_1,\ldots,b_r<p$ and $l<p^r$ then $\binom{b_0+b_1p+\ldots + b_{r-1}p^{r-1} + b_rp^r}{l}=\binom{b_0+b_1p+\ldots + b_{r-1}p^{r-1}}{l}\binom{b_r}{0}=\binom{b_0+b_1p+\ldots + b_{r-1}p^{r-1}}{l}$. In other words, the coefficient of $p^r$ does not matter when computing such a binomial coefficient.
	
	For $\binom{\vh}{l}$, by Equation (\ref{hact2}) we need to show that $\binom{-2ap^r-2z}{l-t}=\binom{-2z}{l-t}$ for all $0\leq t\leq l$. Observing that $\binom{-2ap^r-2z}{l-t}=(-1)^{l-t}\binom{2ap^r+2z+l-t-1}{l-t}$, the previous observation tells us that this is equal to $(-1)^{l-t}\binom{2z+l-t-1}{l-t}$ (unless $z=0$ and $l=t$ in which case the result is trivial) and this is equal to $\binom{-2z}{l-t}$. Hence $\binom{\vh}{l}\phi_a(v_z)=\phi_a(\binom{\vh}{l}v_z)$.
	
	For $\ve^{(l)}$, suppose first that $l<z$. Then Equation (\ref{eact2}) gives $\ve^{(l)}v_z=(\sum_{t=0}^{l}\lambda(\binom{\vh}{t})\binom{l-z}{l-t})v_{z-l}$
	and $\ve^{(l)}v_{ap^r+z}=(\sum_{t=0}^{l}\lambda(\binom{\vh}{t})\binom{l-ap^r-z}{l-t})v_{ap^r + z-l}$. Since $\binom{l-ap^r-z}{l-t}=(-1)^{l-t}\binom{ap^r+z-t-1}{l-t}=(-1)^{l-t}\binom{z-t-1}{l-t}=\binom{l-z}{l-t}$, the result holds. On the other hand, if $l\geq z$ then $\ve^{(l)}v_z=0$ and $\ve^{(l)}v_{ap^r+z}=\lambda(\binom{\vh-ap^r-z+l}{l})v_{ap^r + z-l}$. Since, as in the proof of Lemma~\ref{sssub}, $\lambda(\binom{\vh-ap^r-z+l}{l})\lambda(\binom{\vh}{ap^r+z-l})=\binom{ap^r+z}{l}\lambda(\binom{\vh}{ap^r+z})$, and $\binom{ap^r+z}{l}=0$, we get that either $\lambda(\binom{\vh-ap^r-z+k}{l})=0$ or $\lambda(\binom{\vh}{ap^r+z-l})=0$. The first option clearly gives $\ve^{(l)}v_{ap^r+z}=0$, while the second shows that $\ve^{(l)}v_{ap^r+z}\in  M^r_\chi(\lambda)$ and so is zero in $L^r_\chi(\lambda)$. Thus $\ve^{(l)}\phi_a(v_z)=\phi_a(\ve^{(l)}v_z)$
	
	The result follows.
\end{proof}

\subsection{Non-zero nilpotent $\chi$.}
\label{sec7.3}

In this case, we have $\chi(\vf^{(p^r)})=1$ and, from the definition of $\Lambda^r_\chi$, $\lambda(\binom{\vh}{p^r})^p=\lambda(\binom{\vh}{p^r})$, which hence implies $\lambda(\binom{\vh}{p^r})\in\bF_p$.

For this case we consider the vector subspace of $Z^r_\chi(\lambda)$ with basis
$$\{v_k\,\vert\,\lambda(\binom{\vh}{z})=0\,\,\mbox{where}\,\,k=ap^r+z \,\,\mbox{with}\,\,0\leq z<p^r, 0\leq a<p\}.$$
By a similar argument to the semisimple case, this is the same as 
$$\{v_{a_0+a_1p+\ldots+ a_{r-1}p^{r-1}+a_rp^r}\,\vert\, a_i>\lambda(\binom{\vh}{p^i})\, \mbox{for some}\, i\in\{0,1,\ldots,r-1\}\}.$$
We shall once again denote this subspace by $M^r_\chi(\lambda)$. Again, $M^0_\chi(\lambda)=0$ for all choices of $\chi,\lambda$. 

\begin{lemma}\label{nilsub}
	$M^r_\chi(\lambda)$ is a $U_\chi^{[r]}(G)$-submodule of $Z^r_\chi(\lambda)$.
\end{lemma}

\begin{proof}
	We need to show that this subspace is preserved by $\ve^{(l)}$, $\binom{\vh}{l}$ and $\vf^{(l)}$ for $0\leq l<p^{r+1}$. It is clearly preserved by all $\binom{\vh}{l}$, so we just need to show it for $\ve^{(l)}$ and  $\vf^{(l)}$.
	
	Let $l=a_0+a_1p+\ldots a_{r-1}p^{r-1}+a_rp^r$ and $k=b_0+b_1p+\ldots+b_{r-1}p^{r-1} + b_rp^r$ with $0\leq a_i,b_i<p$ for all $i$.	Then we have $$\vf^{(l)}v_k=\frac{1}{a_0!a_1!\ldots a_{r-1}!a_r!b_0!b_1!\ldots b_{r-1}!b_r!}\vf^{a_0+b_0}(\vf^{(p)})^{a_1+b_1}\ldots(\vf^{(p^{r-1})})^{a_{r-1}+b_{r-1}}(\vf^{(p^r)})^{a_r+b_r}\otimes m_0.$$
	If $a_i+b_i\geq p$ for some $0\leq i<r$ then this expression is zero, since $(\vf^{(p^i)})^{p}=0$. If $a_i+b_i<p$ for all $0\leq i\leq r$, then we have just increased the exponent in each term, which clearly will preserve $M^r_\chi(\lambda)$.
	
	The only remaining case is if $a_i+b_i<p$ for all $0\leq i< r$ and $a_r+b_r= p+s$ for some $0\leq s<p$. In this case, we get from Equation (\ref{fact2}) that $$\vf^{(l)}v_k=\frac{\chi(\vf^{(p^r)})^p}{a_0!a_1!\ldots a_{r-1}!a_r!b_0!b_1!\ldots b_{r-1}!b_r!}\vf^{a_0+b_0}(\vf^{(p)})^{a_1+b_1}\ldots(\vf^{(p^{r-1})})^{a_{r-1}+b_{r-1}}(\vf^{(p^r)})^{s}\otimes m_0.$$ 
	By the second interpretation of the basis this clearly preserves $M^r_\chi(\lambda)$, since we have just increased the exponents in the $\vf^{(p^i)}$ with $i<r$. Hence, we get that $M^r_\chi(\lambda)$ is preserved by the $\vf^{(l)}$. 
	
	Recall from the semisimple case that we have for $l\leq k <p^{r-1}$
	\begin{equation}\label{mult2}
		(\sum_{t=0}^{l}\lambda(\binom{\vh}{t})\binom{l-k}{l-t})\lambda(\binom{\vh}{k-l})=\binom{k}{l}\lambda(\binom{\vh}{k}).
	\end{equation}
	Now let us compute $\ve^{(l)}v_k$ for $k=ap^r+z$ with $0\leq a<p$ and $0\leq z<p^r$ such that $\lambda(\binom{\vh}{z})=0$. When $l>k$ the expression is $0$ and we are fine, so we may assume $l\leq k$. 
	
	First, let us assume that $k<p^r$. So $v_k\in M^r_\chi(\lambda)$ is equivalent to $\lambda(\binom{\vh}{k})=0$. Using the above formula, we get that $$\ve^{(l)}v_k= (\sum_{t=0}^{l}\lambda(\binom{\vh}{t})\binom{l-k}{l-t})v_{k-l}.$$
	Since $\lambda(\binom{\vh}{k})=0$, Equation (\ref{mult}) tells us that either $(\sum_{t=0}^{l}\lambda(\binom{\vh}{t})\binom{l-k}{l-t})=0$, in which case $\ve^{(l)}v_k=0$, or $\lambda(\binom{\vh}{k-l})=0$. Since $k<p^r$, we also have $k-l<p^r$, so $\lambda(\binom{\vh}{k-l})=0$ if and only if $v_{k-l}\in M^r_\chi(\lambda)$, and we are done.
	
	Now suppose that $k=ap^r+ z$ for $0\leq z<p^r$. One can easily check that $\vf^{(k)}=\vf^{(ap^r)}\vf^{(z)}$, so $v_k=\vf^{(ap^r)}v_z$. Hence, we get
	$$\ve^{(l)}v_k=\ve^{(l)}\vf^{(ap^r)}v_z=\sum_{t=0}^{\min(ap^r,l)}\vf^{(ap^r-t)}\binom{\vh-ap^r-l+2t}{t}\ve^{(l-t)}v_z$$
	where $\binom{\vh-ap^r-l+2t}{t}=\sum_{i=0}^{t}\binom{2t-ap^r-l}{t-i}\binom{\vh}{i}$. Since $z<p^r$, we know by the previous case that $\ve^{(l-t)}v_z\in M^r_\chi(\lambda)$ for all $t$ since $\lambda(\binom{\vh}{z})=0$. Hence, since we have already shown that $M^r_\chi(\lambda)$ is preserved by the $\binom{\vh}{i}$ and the $\vf^{(ap^r-t)}$, we get that $\ve^{(l)}v_k=\ve^{(l)}\vf^{(ap^r)}v_z\in M^r_\chi(\lambda)$.
	
	Hence, $M^r_\chi(\lambda)$ is preserved by the $\ve^{(l)}$ and we are done.	
	
\end{proof}

\begin{prop}\label{nilmax}
	$M^r_\chi(\lambda)$ is a maximal submodule of $Z^r_\chi(\lambda)$.
\end{prop}

\begin{proof}
	It is enough to show that the quotient module $L^r_\chi(\lambda)\coloneqq Z^r_\chi(\lambda)/M^r_\chi(\lambda)$ is irreducible. From the above description it is clear that $L^r_\chi(\lambda)$ has as basis the images under the quotient map of the elements 
	$$\{v_k\,\vert\,\lambda(\binom{\vh}{z})\neq 0\,\,\mbox{where}\,\,k=ap^r+z \,\,\mbox{with}\,\,0\leq z<p^r, 0\leq a<p\}.$$
	We shall abuse notation to denote by $v_k$ both the element in $Z^r_\chi(\lambda)$ and its image in the quotient.
	
	Let $v_k$ be an element of this basis, with $k=ap^r+z$ for $0\leq z<p^r$, $0\leq a<p$. Then we have 
	$$\ve^{(z)}v_k=\sum_{t=0}^{z}\lambda(\binom{\vh}{t})\binom{-ap^r}{z-t}v_{ap^r}.$$
	We can calculate that $\binom{-ap^r}{z-t}=(-1)^{z-t}\binom{ap^r+z-t-1}{z-t}$. Hence, for $z\neq t$, we have
	$$\binom{ap^r+z-t-1}{z-t}+\binom{ap^r+z-t-1}{z-t-1}=\binom{ap^r+z-t}{z-t}.$$
	Using Lucas' Theorem, since $z,t<p^r$, we get that $\binom{ap^r+z-t-1}{z-t-1}=1$ and $\binom{ap^r+z-t}{z-t}=1$. This tells us that $\binom{ap^r+z-t-1}{z-t}=0$ when $z\neq t$, and when $z=t$ we clearly get $\binom{ap^r+z-t-1}{z-t}=1$. 
	Hence, $$\ve^{(z)}v_k=\lambda(\binom{\vh}{z})v_{ap^r}\neq 0.$$
	Now, let $N$ be a non-zero submodule of $L^r_\chi(\lambda)$. There hence exists a non-zero element
	$$v=\sum \alpha_t v_t\in N$$
	where the sum is over all $0\leq t<p^{r+1}$ with $\lambda(\binom{\vh}{z})\neq 0$ when $t=ap^r+z$ with $0\leq z<p^r$. Suppose $s=ap^r+z$ with $0\leq z<p^r$ and $0\leq a<p$ is the largest such integer with $\alpha_s\neq 0$. Then we have that the term of $\ve^{(z)}v\in N$ that has the largest index and non-zero coefficient is of the form $v_{ap^r}$ (this has non-zero coefficient since $\lambda(\binom{\vh}{z})\neq 0$). Then either our new vector has a $v_0$ constituent or we can get that $(\vf^{(p^r)})^{p-a}v\in N$ has a $v_0$ constituent (since $\chi(\vf^{(p^r)})^p=1$).
	
	Suppose our vector still contains a term whose index is not divisible by $p^r$. If the new largest term of our vector is divisible by $p^r$, apply $\vf^{(p^r)}$ until the largest term is not divisible by $p^r$, say it has remainder $y<p^r$. Applying $\ve^{(y)}$ will remove the $v_0$ term and make the index of the largest term divisible by $p^r$. Our new vector has fewer terms. 
	
	We can keep applying this process until all the terms in  our vector are divisible by $p^r$, since at each step we are decreasing the number of terms and our vector has only finitely many terms to start with. Furthermore, we never make the largest term zero when we apply our steps. Hence, our final vector cannot be zero. Finally, since $N$ is a submodule, our final vector lies inside $N$. Hence we get that $N$ contains an element of the form 
	$$v=\sum_{i=0}^{p-1}\alpha_iv_{ip^r},$$
	where $\alpha_i\in\bK$ are not all zero. In fact, it is easy to see that we may assume $\alpha_0=1$ by applying powers of $\vf^{(p^r)}$ and rescaling. We can further assume that $\lambda(\binom{\vh}{ip^r})=0$ for all $i\neq 0$ with $\alpha_i\neq 0$ by applying $\ve^{(p^r)}$ enough times.
	
	One can calculate that $\binom{\vh}{p^r}v_{ip^r}=(\lambda(\binom{\vh}{p^r})-2i)v_{ip^r}$ and hence we get
	$$\binom{\vh}{kp^r}v_{ip^r}=(\lambda(\binom{\vh}{p^r})-2i)(\lambda(\binom{\vh}{p^r})-2i-1)\ldots(\lambda(\binom{\vh}{p^r})-2i-k+1)v_{ip^r}.$$
	By applying $\binom{\vh}{kp^r}$ for sufficiently large $k$, we therefore (for $p\neq 2$) end up with a vector $w\in N$ which is a non-zero scalar multiple of a basis element whose index is divisible by $p^r$. Applying $\vf^{(p^r)}$ enough times, we get that $v_0\in N$ and therefore get that $N$ is the whole module.
	
	Hence, we get that $M^r_\chi(\lambda)$ is irreducible.
	
\end{proof}


\begin{prop}\label{niluniq}
	$M^r_\chi(\lambda)$ is the unique maximal submodule of $Z^r_\chi(\lambda)$.
\end{prop}

\begin{proof}
	Suppose that $N$ is a maximal submodule of $Z^r_\chi(\lambda)$ distinct from $M^r_\chi(\lambda)$. 
	
	Since $M^r_\chi(\lambda) + N=Z^r_\chi(\lambda)$, there exists $v\in N$ with $v=v_0 + \sum \alpha_tv_t$, where the sum is over the set of numbers $0\leq t<p^{r+1}$ such that $t=ap^r+z$ with $0\leq a<p$ and $0\leq z<p^r$ with $\lambda(\binom{\vh}{z})=0$.
	
	Let $k_1<k_2<\ldots <k_s$ be the set of integers between $1$ and $p^r-1$ (inclusive) with $\lambda(\binom{\vh}{k_i})=0$. Recall from the proof of Lemma~\ref{nilsub} that if $z=z_0+z_1p+\ldots z_{r-1}p^{r-1} < p^r$ and $k=b_0+b_1p+\ldots+b_{r}p^r$ with $0\leq z_i,b_i<p$ for all $i$ then we have
	$$\vf^{(z)}v_k=\frac{1}{z_0!z_1!\ldots z_{r-1}!b_0!b_1!\ldots b_{r-1}!b_r!}\vf^{z_0+b_0}(\vf^{(p)})^{z_1+b_1}\ldots(\vf^{(p^{r-1})})^{z_{r-1}+b_{r-1}}(\vf^{(p^r)})^{b_r}\otimes m_0.$$
	In particular, we can compute $\vf^{(k_s)}v_{ap^r+z}$. If $z\neq 0$ and $z+k_s<p^r$ then by the maximality of $k_s$ we have $\lambda(\binom{\vh}{z+k_s})\neq 0$ and hence, as $M^r_\chi(\lambda)$ is a submodule, that $\vf^{(k_s)}v_{ap^r+z}=0$. On the other hand, if $z\neq 0$ and $z+k_s\geq p^r$ then, setting $k_{s-1}=b_0+b_1p+\ldots+b_{r}p^r$, we get $z_i+b_i\geq p$ for some $0\leq i \leq {r-1}$, and hence the above formula gives $\vf^{(k_s)}v_{ap^r+z}=0$. Therefore, (as $\lambda(\binom{\vh}{0})=1$) we get that $\vf^{(k_s)}M^r_\chi(\lambda)=0$.	We conclude that $v_{k_s}=\vf^{(k_s)}v_0=\vf^{(k_s)}v\in N$ and hence $\sum_{a=0}^{p-1}\bK v_{ap^r+k_s}\leq N$.
	
	Now, we can compute $\vf^{(k_{s-1})}v_{ap^r+z}$. If $z\neq 0$ and $z+k_s<p^r$ then by the maximality of $k_s$ we either have that  $\lambda(\binom{\vh}{z+k_{s-1}})\neq 0$ and, as $M^r_\chi(\lambda)$ is a submodule, that $\vf^{(k_{s-1})}v_{ap^r+z}=0$, or that $z+k_{s-1}=k_s$ and that $\vf^{(k_{s-1})}v_{ap^r+z}\in \bK\{v_{ap^r+k_s}\,\vert\,0\leq a<p\}\subset N$. On the other hand, if $z\neq 0$ and $z+k_{s-1}\geq p^r$ then, now setting $k_{s-1}=b_0+b_1p+\ldots+b_{r}p^r$, we get $z_i+b_i\geq p$ for some $0\leq i \leq {r-1}$, and hence the above formula gives $\vf^{(k_{s-1})}v_{ap^r+z}=0$. Therefore, (as $\lambda(\binom{\vh}{0})=1$) we get that $\vf^{(k_{s-1})}M^r_\chi(\lambda)\subset N$. We conclude that $v_{k_{s-1}}=\vf^{(k_{s-1})}v_0=\vf^{(k_{s-1})}v-\sum \alpha_t\vf^{(k_{s-1})}v_t\in N$.
	
	An inductive argument gives that $v_{k_i}\in N$ for all $1\leq i\leq s$, and hence that $M^r_\chi(\lambda)\subset N$. This contradiction tells us that $M^r_\chi(\lambda)$ is the unique maximal submodule of $Z^r_\chi(\lambda)$.
\end{proof}

Hence, we have constructed all irreducible $U_\chi^{[r]}(SL_2)$-modules in the case when $\chi\neq 0$ is nilpotent.
Given $\lambda\in\Lambda^r_\chi$, we hence get a unique irreducible $U_\chi^{[r]}(SL_2)$-module of dimension:
$$(\lambda(\vh)+1)(\lambda(\binom{\vh}{p})+1)\ldots (\lambda(\binom{\vh}{p^{r-1}})+1)p$$ where we view $\lambda(\binom{\vh}{p^i})$ as elements of the set $\{0,1,2,\ldots,p-1\}$ for $0\leq i<r$.

Again, we can also say something about the structure of these irreducible modules as $\Di(G_{(r)})$ modules.

\begin{prop}\label{DistDec2}
	Each irreducible $U_{\chi}^{[r]}(G)$-module $M$ decomposes as $\Di(G_{(r)})$-modules into a direct sum of $p$ copies of the same irreducible $\Di(G_{(r)})$-module. In particular, if $M$ is a quotient of $Z_\chi^r(\lambda)$ for $\lambda\in \Lambda_\chi^r$ then this irreducible $\Di(G_{(r)})$-module is the module $L_r(\lambda)$ defined in \cite{Jan}.
\end{prop}

\begin{proof}
	
	By Lemma~\ref{nilsub} and Propositions~\ref{nilmax} and \ref{niluniq}, we know that each irreducible $U_\chi^{[r]}(G)$-module is obtained as the unique irreducible quotient module of $Z_\chi^r(\lambda)$ for some $\lambda\in\Lambda_\chi^r$ and has as basis the images under the quotient map of
	$$\{v_k\,\vert\,\lambda(\binom{\vh}{z})\neq 0\,\,\mbox{where}\,\,k=ap^r+z \,\,\mbox{with}\,\,0\leq z<p^r, 0\leq a<p\}$$
	or equivalently those of
	$$\{v_{a_0+a_1p+\ldots+ a_{r-1}p^{r-1}+a_rp^r}\,\vert\, 0\leq a_i\leq\lambda(\binom{\vh}{p^i})\, \mbox{for all}\, i\in\{0,1,\ldots,r-1\},\,\, 0\leq a_r< p\}.$$
	We shall denote this module as $L_\chi^r(\lambda)$. We shall, as before, abuse notation to denote by $v_k$ both the element in $Z^r_\chi(\lambda)$ and its image in $L_\chi^r(\lambda)$.
	
	We proceed as in the proof of Proposition \ref{DistDec}. The only part of the proof that has to be modified for the nilpotent case is the proof that $\ve^{(l)}\phi_a(v_z)=\phi_a(\ve^{(l)}v_z)$, for $0\leq l,z<p^r$ and $1\leq a<p$. When $l<z$ the result follows as in Proposition \ref{DistDec}. When $l\geq z$ we still get that $\ve^{(l)}v_z=0$ and $\ve^{(l)}v_{ap^r+z}=\lambda(\binom{\vh-ap^r-z+l}{l})v_{ap^r + z-l}$, but now we need to observe that $\lambda(\binom{\vh-ap^r-z+l}{l})\lambda(\binom{\vh-(a-1)p^r}{p^r+z-l})=\lambda(\binom{\vh-(a-1)p^r}{p^r+z})\binom{p^r+z}{l}$ (this comes from the similar result over $\bC$, using the $\bZ_{(p)}$-form). Since $\binom{p^r+z}{l}=0$, we have that either $\lambda(\binom{\vh-ap^r-z+l}{l})=0$, in which case we are done, or $\lambda(\binom{\vh-(a-1)p^r}{p^r+z-l})=0$.
	
	We now observe that $\lambda(\binom{\vh-(a-1)p^r}{p^r+z-l})=\sum_{t=0}^{p^r+z-l}\lambda(\binom{\vh}{t})\binom{-(a-1)p^r}{p^r+z-l-t}$ and that $\binom{-(a-1)p^r}{p^r+z-l-t}=(-1)^{p^r+z-l-t}\binom{ap^r+z-l-t-1}{p^r+z-l-t}$. If $t\neq p^r+z-l$ then Lucas' Theorem gives that $\binom{ap^r+z-l-t}{p^r+z-l-t}=1$ and $\binom{ap^r+z-l-t-1}{p^r+z-l-t-1}=1$, so since $\binom{ap^r+z-l-t-1}{p^r-l-t}+\binom{ap^r+z-l-t-1}{p^r+z-l-t-1}=\binom{ap^r+z-l-t}{p^r+z-l-t}$ we get that $\binom{ap^r+z-l-t-1}{p^r-l-t}=0$ for all $t\neq p^r+z-l$. In particular, $\lambda(\binom{\vh-(a-1)p^r}{p^r+z-l})=\sum_{t=0}^{p^r+z-l}\lambda(\binom{\vh}{t})\binom{-(a-1)p^r}{p^r+z-l-t}=\lambda(\binom{\vh}{p^r+z-l})$.
	
	Hence, $\ve^{(l)}v_{ap^r+z}=\lambda(\binom{\vh-ap^r-z+l}{k})v_{ap^r + z-l}\in M^r_\chi(\lambda)$ and its image in $L^r_\chi(\lambda)$ is zero, as required. The result follows as in Proposition \ref{DistDec}.
	
\end{proof}


\subsection{Zero $\chi$}
\label{sec7.4}

Recall that $U_0^{[r]}(G)\cong \Di(G_{(r+1)})$. This case has hence been well-studied before (see, for example, Chapter II.3 in \cite{Jan}), but let us understand how the results can be derived in our context. 

As in the semisimple case, we take 
$$M^r_\chi(\lambda)\coloneqq\bK\{v_{k}\,\vert\,\lambda(\binom{\vh}{k})=0\,\mbox{and}\,\,0\leq k<p^{r+1}\}.$$
The main difference in this case is that this vector space is now equal to 
$$\{v_{a_0+a_1p+\ldots+ a_{r-1}p^{r-1}+a_rp^r}\,\vert\, a_i>\lambda(\binom{\vh}{p^i})\, \mbox{for some}\, i\in\{0,1,\ldots,r\}\}$$
since $\lambda(\binom{\vh}{p^j})\in\bF_p$ for all $0\leq j\leq r$.

The proofs in the semisimple case now work almost exactly the same in the $\chi=0$ case. Hence we get that, for each $\lambda\in\Lambda^{r}_0$, we have a unique irreducible $U_0^{[r]}(G)$-module, and this module has dimension
$$(\lambda(\vh)+1)(\lambda(\binom{\vh}{p})+1)\ldots (\lambda(\binom{\vh}{p^{r}})+1),$$
where we view $\lambda(\binom{\vh}{p^i})$ as elements of the set $\{0,1,2,\ldots,p-1\}$ for $0\leq i\leq r$.

Let us now examine the structure of the irreducible $U_\chi^{[r]}(G)$-modules when we consider them as $\Di(G_{(r)})$-modules.


\begin{prop}\label{DistDec3}
	Each irreducible $U_{\chi}^{[r]}(G)$-module decomposes as $\Di(G_{(r)})$-modules into a direct sum of $\lambda(\binom{\vh}{p^r})+1$ copies of the same irreducible $\Di(G_{(r)})$-module. In particular, if $M$ is a quotient of $Z_\chi^r(\lambda)$ for $\lambda\in \Lambda_\chi^r$ then this irreducible $\Di(G_{(r)})$-module is the module $L_r(\lambda)$ defined in \cite{Jan}.
\end{prop}

\begin{proof}
	Very similar to Propositions \ref{DistDec} and \ref{DistDec2}, adapted for the appropriate basis. Details are left to the interested reader.
\end{proof}

\subsection{Classification}

Now that we know the irreducible modules for $U_\chi^{[r]}(G)$, we can reinterpret them using a different construction. Define the subalgebra $\widehat{U^{[r]}_\chi(B)}$ of $U^{[r]}(G)$ as the subalgebra generated by $\Di(G_{(r)})$ and $U_\chi^{[r]}(B)$, which has basis $\{\vf^{(i)}\binom{\vh}{k}\ve^{(j)}\,\vert\,0\leq i<p^r\,\mbox{and}\,0\leq j,k<p^{r+1}\}$. Suppose that $N$ is an irreducible $\Di(G_{(r)})$-module, coming from $\lambda_{r-1}\in\Lambda_0^{r-1}$. Note that $N$ has a unique (up to scalar multiplication) highest weight vector $v_0$ \cite[II.3.10]{Jan}, and so has a basis consisting of the non-zero elements of $\{\vf^{(i)}v_0\,\vert\,0\leq i<p^{r}\}$. By choosing an extension of $\lambda_{r-1}$ to $\lambda_r\in\Lambda_\chi^r$, we can extend the $\Di(G_{(r)})$-module structure on $N$ to a $\widehat{U^{[r]}_\chi(B)}$-module structure on $N$ by letting $\ve^{(p^r)}$ act as $0$ and $\binom{\vh}{p^r}$ act by scalar multiplication as it does on the basis elements $\{\vf^{(i)}v_0\}_{0\leq i<p^{r}}$ of the higher baby Verma modules (which depends on $\lambda_r$).

We can then define the {\em teenage Verma module} $Z_\chi^{[r]}(N,\lambda_r)$ as 
$$Z_\chi^{[r]}(N,\lambda_r)\coloneqq U_\chi^{[r]}(G)\otimes_{\widehat{U^{[r]}_\chi(B)}}N.$$


\begin{theorem}[Classification of irreducible $U_\chi^{[r]}(SL_2)$-modules]\label{Class}
	We have the following classification of irreducible $U_\chi^{[r]}(SL_2)$-modules, for $(\chi\in\fs\fl_2^{*})^{[r]}$:
	\begin{itemize}
		\item If $\chi\neq 0$ is semisimple, then the irreducible modules are the $Z_\chi^{[r]}(N,\lambda_r)$ for $N$ an irreducible $\Di(SL_{2,r})$-module with weight $\lambda_{r-1}\in\Lambda_0^{r-1}$, and $\lambda_r\in\Lambda_\chi^{r}$ extending $\lambda_{r-1}$. Furthermore, these are all non-isomorphic, so there are exactly $p^{r+1}$ non-isomorphic $U_\chi^{[r]}(SL_2)$-modules.
		\item If $\chi\neq 0$ is nilpotent, then the irreducible modules are the $Z_\chi^{[r]}(N,\lambda_r)$ for $N$ an irreducible $\Di(SL_{2,r})$-module with weight $\lambda_{r-1}\in\Lambda_0^{r-1}$, and $\lambda_r\in\Lambda_\chi^{r}$ extending $\lambda_{r-1}$. Furthermore, $Z_\chi^{[r]}(N,\lambda_r)=Z_\chi^{[r]}(M,\lambda'_r)$ if and only if $N=M$ and $\lambda_r=\lambda_r'$ or  $\lambda'_r(\binom{\vh}{p^r})=p-\lambda_r(\binom{\vh}{p^r})-2$ and $\lambda_r(\binom{\vh}{p^r})\leq p-2$ (as an element of $\{0,1,\ldots,p-1\}$), so there are exactly $p^r(\frac{p+1}{2})$ non-isomorphic $U_\chi^{[r]}(SL_2)$-modules.
		\item If $\chi=0$, every irreducible $U_0^{[r]}(SL_2)$ is the unique irreducible quotient of $Z_\chi^{[r]}(N,\lambda_r)$ for $N$ an irreducible $\Di(SL_{2,r})$-module with weight $\lambda_{r-1}\in\Lambda_0^{r-1}$ and $\lambda_r\in\Lambda_\chi^{r}$ extending $\lambda_{r-1}$.
	\end{itemize}
\end{theorem}

\begin{proof}
	Most of this theorem is proved in Sections \ref{sec7.2}, \ref{sec7.3} and \ref{sec7.4}. All that remains is to prove the isomorphism conditions. Observe that Propositions \ref{DistDec} and \ref{DistDec2} show that for $\chi\neq 0$ we have that $Z_\chi^{[r]}(N,\lambda_r)\ncong Z_\chi^{[r]}(M,\lambda'_r)$ for distinct irreducible $\Di(G_{(r)})$-modules $N$ and $M$, independent of the choice of $\lambda_r$ and $\lambda'_r$.
	
	When $\chi\neq 0$ is semisimple, one can show using Proposition~\ref{DistDec} and methods similar to those used in the rest of this chapter, that $$\{v\in Z_\chi^{[r]}(N,\lambda_r)\,\vert\,\ve^{(p^r)}v=0\,\}=\bK\{v_i\,\vert\,i<p^r, \lambda_r(\binom{\vh}{i})\neq 0\}.$$ Letting $z$ be the largest integer less than to $p^r$ such that $\lambda_r(\binom{\vh}{z})\neq 0$, we hence have that $$\ve^{(z)}\{v\in Z_\chi^{[r]}(N,\lambda_r)\,\vert\,\ve^{(p^r)}v=0\,\}=\bK v_0.$$ In particular, $\lambda_r$ is determined by the action of the $\binom{\vh}{k}$ on this subspace.
	
	When $\chi\neq 0$ is nilpotent, fix $a=\lambda_r(\binom{\vh}{p^r})\in\bF_p=\{0,1,\ldots,p-1\}$. One can similarly show that $$\{v\in Z_\chi^{[r]}(N,\lambda_r)\,\vert\,\ve^{(p^r)}v=0\,\}=\bK\{v_i,\, v_{ap^r+i}\vert\,i<p^r, \lambda_r(\binom{\vh}{i})\neq 0\,\}.$$ Letting $z$ be the largest integer less than to $p^r$ such that $\lambda_r(\binom{\vh}{z})\neq 0$, we hence have that $$\ve^{(z)}\{v\in Z_\chi^{[r]}(N,\lambda_r)\,\vert\,\ve^{(p^r)}v=0\,\}=\bK v_0 + \bK v_{\lambda_r(\binom{\vh}{p^r})p^r}.$$ By calculation, one can check that the line $\bK v_{\lambda_r(\binom{\vh}{p^r})p^r}$ is a $U_\chi^{[r]}(B)$-module isomorphic  to $\bK_{\mu_r}$ where $\mu_r(\binom{\vh}{p^i})=\lambda_r(\binom{\vh}{p^i})$ for $i<r$ and $\mu_r(\binom{\vh}{p^r})=p-\lambda_r(\binom{\vh}{p^r})-2$. In particular, $Z_\chi^{[r]}(N,\lambda_r)\cong Z_\chi^{[r]}(N,\mu_r)$. On the other hand, the description of $\ve^{(z)}\{v\in Z_\chi^{[r]}(N,\lambda_r)\,\vert\,\ve^{(p^r)}v=0\,\}$ guarantees that these are the only possible (non-identity) isomorphisms.
	The uniqueness when $\chi=0$ is well known (see \cite{Jan}).
	
\end{proof}

\begin{cor}
	Let $(\chi\in\fs\fl_2^{*})^{[r]}$ and let $\lambda,\mu\in\Lambda_\chi^r$ with $\mu\neq \lambda$.	Keeping the notation from Subsections~\ref{sec7.2}, \ref{sec7.3} and \ref{sec7.4}, let $L_\chi^r(\lambda)$ and $L_\chi^r(\mu)$ be the corresponding irreducible $U^{[r]}_\chi(SL_2)$-modules. If $\chi$ is zero or non-zero semisimple then $L_\chi^r(\lambda)\ncong L_\chi^r(\mu)$. If $\chi$ is non-zero nilpotent, then $L_\chi^r(\lambda)\cong L_\chi^r(\mu)$ if and only if $\mu(\binom{\vh}{p^j})=\lambda(\binom{\vh}{p^j})$ for all $j<r$ and $\mu(\binom{\vh}{p^r})=-\lambda(\binom{\vh}{p^r})-2$.
\end{cor}

\subsection{Conjectures}
\label{sec8}

Based on our understanding of the case in Section~\ref{sec7} and of the $r=0$ case (see \cite{Jan2}), we can formulate some conjectures about the representation theory of $U_\chi^{[r]}(G)$.

\begin{conj}
	Let $N$ be an irreducible $\Di(G_{(r)})$-module with corresponding weight $\lambda_{r-1}\in \Lambda_0^{r-1}$. Let $\widehat{U_\chi^{[r]}(B)}$ be the subalgebra of $U^{[r]}_\chi(G)$ generated by $\Di(G_{(r)})$ and $U_\chi^{[r]}(B)$. Then each extension of $\lambda_{r-1}$ to $\lambda_r\in\Lambda_\chi^r$ determines an irreducible $\widehat{U_\chi^{[r]}(B)}$-module structure on the $\Di(G_{(r)})$-module $N$, and every irreducible $\widehat{U_\chi^{[r]}(B)}$-module restricts to an irreducible $\Di(G_{(r)})$-module.
\end{conj}

A proof of this result would immediately lead to a proof of the following conjecture, an analogue of the result from the $r=0$ case that every irreducible $\fg$-module is the quotient of a baby Verma module.

\begin{conj}\label{teen}
	Every irreducible $U_\chi^{[r]}(G)$-module is a homomorphic image of $Z_\chi^{[r]}(N,\lambda_r)\coloneqq U_\chi^{[r]}(G)\otimes_{\widehat{U_\chi^{[r]}(B)}}N$ for some irreducible $\Di(G_{(r)})$-module $N$ and $\lambda_{r}\in\Lambda_\chi^{r}$ extending the weight of $N$, where $N$ is given the structure of a $\widehat{U_\chi^{[r]}(B)}$-module as in the previous conjecture.
\end{conj}

These conjectures are proved in the sequel \cite{West}.

\section{Hopf Algebra Structure}\label{sec9}

\subsection{Hopf subalgebra structure}

Corollary~\ref{Hopf} tells us that, much like the universal enveloping algebra and distribution algebra, the higher universal enveloping algebras $\Upr$ have the structure of cocommutative Hopf algebras. This Hopf-algebraic structure is not used substantially in the rest of this paper, but is a key focus in the sequel \cite{West}. Nonetheless, in the case of reductive groups, it is worthwhile to take a moment and use the results of this paper to derive some Hopf-algebraic properties of the higher universal enveloping algebras.

We start with some important observations.

\begin{lemma}
	For a reductive algebraic group $G$, the algebra $\Upr$ satisfies the following properties:
	\begin{enumerate}
		\item $\Di(G_{(r)})$ is a normal Hopf subalgebra of $\Upr$.
		\item $\Upr$ is free as a left and right $\Di(G_{(r)})$-module.
		\item $\Upr$ is faithfully flat as a left and right $\Di(G_{(r)})$-module.
		\item $\Upr/\Di^{+}(G_{(r)})\Upr$ is isomorphic to the Hopf algebra $U(\fg)$.
		\item $\Di(G_{(r)})\subset\Upr$ is a $U(\fg)$-Galois extension, with $\Di(G_{(r)})=\Upr^{co U(\fg)}$.
	\end{enumerate}
\end{lemma}

\begin{proof}
	Recall that a Hopf subalgebra $B$ of a Hopf algebra $A$ is said to be \emph{normal} if $\ad_l(a)(b)\in B$ and $\ad_r(a)(b)\in B$ for all $a\in A$ and $b\in B$ where, using Sweedler's $\Sigma$-notation, $$\ad_l(a)(b)=\sum a_{(1)}bS(a_{(2)}),\quad\,\,\,\,\ad_r(a)(b)=\sum S(a_{(1)})ba_{(2)}.$$
	
	Since $\Upr$ is cocommutative, these two conditions are equivalent, so it is enough to prove closure under the left adjoint. Since $\ad_l(aa')(b)=\ad_l(a)\ad_l(a')(b)$ and $\ad_l(a)(bb')=\sum (\ad_l(a_{(1)})b)(\ad_l(a_{(2)})b')$ for $a,a'\in A$ and $b,b'\in B$, it is enough to show closure for generators of $A$ and $B$. When $G$ is reductive, $\Di(G_{(r)})\subset \Di(G)$ is generated by $\Di_{p^r-1}(G)$ and $\Upr$ is generated by $\Di_{p^r}(G)$. Let $\delta\in \Di_{p^r}(G)$ and $\mu\in \Di_{p^r-1}(G)$. Then
	$$\ad_l(\delta)(\mu)=\sum \delta_{(1)}\otimes \mu\otimes S(\delta_{(2)}),$$
	where the $\otimes$ represents the multiplication in $\Upr$, and we have $\delta_{(1)}\in\Di_i(G)$, $\delta_{(2)}\in\Di_j(G)$ with $i+j=p^r$. In particular, $i+p^r-1+j<p^{r+1}$ and so in fact 
	$$\ad_l(\delta)(\mu)=\sum \delta_{(1)} \mu S(\delta_{(2)}),$$
	with the multiplication now in $\Di_{p^{r+1}-1}(G)$, the restriction of the multiplication in $\Di(G)$. Since $\Di(G_{(r)})$ is normal in $\Di(G)$ \cite[I.7.18]{Jan}, we hence conclude that $\ad_l(\delta)(\mu)\in\Di(G_{(r)})$. This proves (1).
	
	Part (2) then follows from Theorem 2.1(2) in \cite{Sch2}, and (3) follows from (2). Furthermore, (4) is easy to see from the results of Section~\ref{sec4}, and (5) follows from Remark 1.1(4) in \cite{Sch}.
\end{proof}

Recall that $X(T)=\Hom(T,\bG_m)$ is the character group of $T$, where $\bG_m$ is the multiplicative group and $T$ is a maximal torus of $G$. Let $Y(T)=\Hom(\bG_m,T)$ be the cocharacter group of $G$. Then, as in \cite[II.1.3]{Jan}, there exists a bilinear pairing $X(T)\times Y(T)$ given by $(\lambda, \mu)\mapsto \langle\lambda,\mu\rangle$, where $\langle\lambda,\mu\rangle$ is the integer corresponding to $\mu\circ\lambda\in\End(\bG_m)=\bZ$.

As always, we have $\Phi$ is the root system of $G$ with respect to $T$, and we choose $\Phi^{+}$ a system of positive roots and $\Pi$ a set of simple roots inside $\Phi^{+}$. Given $\alpha\in \Phi$, we define $\alpha^{\nu}\in Y(T)$ be the coroot of $\alpha$. We hence define
$$X(T)_{+}\coloneqq\{\lambda\in X(T)\vert\,\langle \lambda,\alpha^{\nu}\rangle\geq 0\;\mbox{for all}\;\alpha\in\Pi\}$$
to be the set of {\em dominant weights} of $T$ with respect to $\Phi^{+}$ and, for $r\geq 1$, we set 
$$X_r(T)\coloneqq\{\lambda\in X(T)\vert\,0\leq\langle \lambda,\alpha^{\nu}\rangle< p^r \;\mbox{for all}\;\alpha\in\Pi\}.$$

Throughout this section we shall make the assumption that the abelian group $X(T)/p^rX(T)$ has a set of representatives $X_r'(T)$ with $X_r'(T)\subset X_r(T)$. We shall call this assumption (R).

Now, suppose $N$ is an irreducible left $\Di(G_{(r)})$-module and $M$ is an irreducible left $U^{[r]}(G)$-module. Since $N$ is an irreducible left $\Di(G_{(r)})$-module it is a left $\Di(G_{(r+1)})$-module by Proposition II.3.15 in \cite{Jan} (using assumption (R), the fact that $X_r(T)\subset X_{r+1}(T)$, and the fact that the irreducible $\Di(G_{(r)})$-modules are indexed by $X'_r(T)$ - see \cite[II.3.10]{Jan}). Hence, as $U^{[r]}(G)$ surjects onto $\Di(G_{(r+1)})$, $N$ can be extended to a $U^{[r]}(G)$-module.

We can also define a left $U^{[r]}(G)$-module structure on $\Hom_D(N,M)$ as follows (defining $D\coloneqq \Di(G_{(r)})$ for ease of notation):
$$x\cdot\phi:n\mapsto \sum x_{(1)}\phi(S(x_{(2)})n)\qquad \mbox{for}\quad x\in U^{[r]}(G),\,n\in N,\, \phi\in \Hom_{\tiny D}(N,M),$$
where here we are using the $U^{[r]}(G)$-module structure on $N$ defined in the previous paragraph. It is a straightforward calculation that this makes $\Hom_{\tiny D}(N,M)$ into a $U^{[r]}(G)$-module, and that the ideal $\Upr\Di^{+}(G_{(r)})$ acts trivially upon it. Hence, $\Hom_{\tiny D}(N,M)$ has the structure of a module over  $$U(\fg)=\frac{\Upr}{\Upr\Di^{+}(G_{(r)})}.$$

Putting these two observations together and again using the Hopf algebra structure of $\Upr$, we can define a $U^{[r]}(G)$-module structure on $N\otimes\Hom_{\tiny \Di(G_{(r)})}(N,M)$. Furthermore, if $x\in\Di(G_{(r)})$, $n\in N$ and $\phi\in \Hom_{\tiny D}(N,M)$, then 
\begin{equation}\label{GrAct}x\cdot(n\otimes \phi)=\sum x_{(1)}n\otimes x_{(2)}\phi=\sum x_{(1)}n\otimes \epsilon(x_{(2)})\phi=(\sum x_{(1)}\epsilon(x_{(2)})n)\otimes \phi)=xn\otimes \phi,\end{equation}
using here that elements of $\Di(G_{(r)})$ act on $\Hom_{\tiny D}(N,M)$ via $\epsilon$, the counit. So we see that the $\Upr$-module structure on $N\otimes\Hom_{\tiny \Di(G_{(r)})}(N,M)$ restricts to the $\Di(G_{(r)})$-module structure on copies of $N$.

The following result should be compared to Propositions ~\ref{DistDec}, \ref{DistDec2} and \ref{DistDec3} above.

\begin{theorem}\label{Decomp}
	Suppose assumption (R) holds. Let $M$ be an irreducible $U^{[r]}(G)$-module. Then there exists an irreducible $\Di(G_{(r)})$-module $N$ such that $M\cong N\otimes\Hom_{\tiny \Di(G_{(r)})}(N,M)$ as $U^{[r]}(G)$-modules.
\end{theorem}

\begin{proof}
	Let $N$ be an irreducible $\Di(G_{(r)})$-submodule of $M$. As above can then equip $N\otimes\Hom_{\tiny D}(N,M)$ with the structure of a $\Upr$-module.
	
	We define the map $$\Psi:N\otimes\Hom_{\tiny D}(N,M)\to M,\qquad \,\Psi(n\otimes \phi)=\phi(n).$$
	It is straightforward to check that this is a homomorphism of $\Upr$-modules. Since $M$ is irreducible, it is clearly surjective. Hence, using Equation (\ref{GrAct}), as $\Di(G_{(r)})$-modules $$M\cong\bigoplus_{i=1}^k N,$$
	for some $k\in\bN$. In particular, this implies that $\Hom_{\tiny D}(N,M)\cong\bK^{k}$ and so $\dim_{\bK}(M)=k\dim_{\bK}(N)$. Furthermore, $\dim_{\bK}(N\otimes \Hom_{\tiny D}(N,M))=k\dim_{\bK}N$. Hence, $\Psi$ is an isomorphism.
\end{proof}

So Theorem~\ref{Decomp} shows that an irreducible $U^{[r]}(G)$-module can be decomposed into an irreducible $\Di(G_{(r)})$-module and a $U(\fg)$-module.

There is another way to obtain this result. This alternative method shall be more useful in the sequel \cite{West}, and is inspired by papers of Schneider \cite{Sch} and Witherspoon \cite{Wit}.

\begin{lemma}\label{tens}
	Suppose assumption (R) holds. Let $N$ be an irreducible left $\Di(G_{(r)})$-module, and define the algebra $E\coloneqq\End_{\Upr}(\Upr\otimes_{\tiny D}N)^{op}$. Let $U$ be an irreducible left $E$-module. Then $N\otimes_{\bK} U$ can be given a left $U^{[r]}(G)$-module structure which restricts to the natural left $\Di(G_{(r)})$-module structure.
\end{lemma}

\begin{proof}
	The proof of this lemma can essentially be found in \cite{Wit}, but we include elements of it here for completeness. As described above, $N$ can be extended to a $U^{[r]}(G)$-module. Remark 3.2(3) of \cite{Sch} shows that $N$ is $\Upr$-stable (i.e. there is a left $\Di(G_{(r)})$-linear and right $U(\fg)$-collinear isomorphism $\Upr\otimes_{\tiny \Di(G_{(r)})}N\cong N\otimes_{\bK} U(\fg)$ - see, for example, \cite{Sch} or \cite{Wit} for the $U(\fg)$-comodule structures on these spaces).
	
	It was proved in \cite{Sch3} that $N\otimes_{\bK} E$ is isomorphic to $U^{[r]}(G)\otimes_{\tiny \Di(G_{(r)})} N$ as right $E$-modules, using the $U^{[r]}(G)$-stability of $N$. In particular, by applying $-\otimes_E U$, this implies that \begin{equation}\label{DiffMod}
		N\otimes_{\bK} U\cong (U^{[r]}(G)\otimes_{\tiny \Di(G_{(r)})} N)\otimes_E U\end{equation} can be given the structure of a left $U^{[r]}(G)$-module. Theorem 2.2(i) of \cite{Wit} further shows that this $U^{[r]}(G)$-module structure restricts to the natural $\Di(G_{(r)})$-module structure (although this theorem is not directly applicable to this setting, Witherspoon observed in \cite{Wit} that the result still holds in this situation).
\end{proof}

\begin{rmk}\label{Decomp2}
	Lemma~\ref{tens} gives another way to get a $U^{[r]}(G)$-module structure on $N\otimes\Hom_{\tiny \Di(G_{(r)})}(N,M)$, using the observation that $\Hom_{\tiny \Di(G_{(r)})}(N,M)$ is a left $E$-module (see, for example Theorem 2.2.(ii) in \cite{Wit}). 
\end{rmk}

The key point of the proof of Lemma~\ref{tens} is Equation (\ref{DiffMod}), which in our context gives an isomorphism of $\Upr$-modules
$$N\otimes \Hom_{\tiny D}(N,M)\cong(\Upr\otimes_{\tiny D} N)\otimes_{E}\Hom_{\tiny \Upr}(\Upr\otimes_{\tiny D}N,M).$$
It is straightforward to show that the map
$$\eta_M:(\Upr\otimes_{\tiny D} N)\otimes_{E}\Hom_{\tiny \Upr}(\Upr\otimes_{\tiny D}N,M)\to M,\qquad \eta_M(a\otimes_{\tiny D}n\otimes_{E}\phi)=\phi(a\otimes_{\tiny D}n)$$ is a $\Upr$-module homomorphism, and a similar argument to Theorem~\ref{Decomp} shows that it is an isomorphism. So we once again obtain the result:

\begin{theorem}\label{Decomp3}
	Suppose assumption (R) holds. Let $M$ be an irreducible $U^{[r]}(G)$-module. Then there exists an irreducible $\Di(G_{(r)})$-module $N$ such that $M\cong N\otimes\Hom_{\tiny \Di(G_{(r)})}(N,M)$ as $U^{[r]}(G)$-modules, where the $U^{[r]}(G)$-module structure on $N\otimes\Hom_{\tiny \Di(G_{(r)})}(N,M)$ comes from Lemma~\ref{tens}.
\end{theorem}

\begin{rmk}
	Partial credit for this proof and that of Theorem~\ref{prem} below goes to Dmitriy Rumynin, who was kind enough to share it with the author.
\end{rmk}

We observed above that $\Hom_{\tiny \Di(G_{(r)})}(N,M)$ is a left $E$-module. While at first blush the algebra $E$ may appear strange, it turns out to be an algebra we know very well, as the following lemma shows.

\begin{lemma}\label{End}
	Suppose assumption (R) holds. Let $N\in\Irr(\Di(G_{r}))$ and $E=\End_{U^{[r]}(G)}(U^{[r]}(G)\otimes_{\tiny \Di(G_{r})} N)^{op}$. Then $E\cong U(\fg)$.
\end{lemma}

\begin{proof}
	As above, we can observe that the $\Di(G_{r})$-module $N$ can be extended to a $U^{[r]}(G)$-module. Remark 3.8 in \cite{Sch} then tells us that $\bK\subset E$ is a trivial $U(\fg)$-crossed product, and hence that $E\cong \bK\# U(\fg)\cong U(\fg)$.
\end{proof}

\begin{rmk}
	The exact nature of this isomorphism shall be explored more in the sequel \cite{West}. The results there will show, in particular, that the action of $U(\fg)$ on $\Hom_{\tiny \Di(G_{(r)})}(N,M)$ through the quotient \newline $\Upr/\Upr\Di^{+}(G_{(r)})$ and the action of $E$ on $\Hom_{\tiny \Di(G_{(r)})}(N,M)$ described above are compatible with the isomorphism in Lemma~\ref{End}.
\end{rmk}

So we get another way of seeing that an irreducible $U^{[r]}(G)$-module can be decomposed into an irreducible $\Di(G_{(r)})$-module and a $U(\fg)$-module.

What is the benefit of this latter method of proof? Essentially, the initial approach uses the Hopf algebra structure of $U^{[r]}(G)$ to give certain vector spaces a module structure, while the latter approach uses the Hopf algebra structure to get an isomorphism $U(\fg)\cong E$ and then uses just the \emph{algebra} structures to define the modules. Once one knows such an isomorphism exists, it is often-times easier in practice to work with an action which only depends on the algebra structure rather than an action which depends on the whole Hopf algebra structure. 

For example, the second approach means that given a left $U(\fg)$-module $U$ and left $\Di(G_{(r)})$-module $N$, the equation $$N\otimes_{\bK} U\cong (U^{[r]}(G)\otimes_{\tiny \Di(G_{(r)})} N)\otimes_E U$$ allows us to write the $\Upr$-action down very easily. This will have particular use when considering the action of central elements of $\Upr$, such as elements of the form $\delta^{\otimes p}-\delta^p$. Furthermore, the action on $E$ on $\Hom_{\tiny \Di(G_{(r)})}(N,M)$ is often easier to calculate with than the action of $U(\fg)$ on the same.

Let us now consider an application of the decomposition described above. Recall that in the $r=0$ case we have Premet's theorem, under some weak conditions on $p$ and $\fg$ (see \cite{Prem} for details):

\begin{theorem}[Premet's Theorem \cite{Prem}]
	Let $\fg$ and $p$ be as above. Let $\chi\in\fg^{*}$, and let $M$ be a $U_\chi(\fg)$-module. Then $p^{\dim(G\cdot\chi)/2}$ divides $\dim M$.
\end{theorem}

Observe that the natural extension of this theorem would be that $p^{(r+1)\dim(G\cdot\chi)/2}$ divides $\dim M$ for any $U_\chi^{[r]}(G)$-module $M$. We know that this fails in $G=SL_2$. Furthermore, since any $U_\chi(\fg)$-module can be made into a $U_\chi^{[r]}(G)$-module, this extension will not hold for any $G$ we are interested in. 

Theorem~\ref{Decomp} suggests a way to generalise Premet's theorem for the higher universal enveloping algebras $\Upr$. We announce this proposition here, although defer a key part of proof to the sequel \cite{West}, where some necessary infrastructure will be developed. 

\begin{prop}\label{prem}
	Suppose that $G$ is a connected reductive algebraic group over an algebraically closed field $\bK$ of positive characteristic $p>0$ such that assumption (R) holds. Suppose further that $\fg$ and $p$ are such that Premet's theorem holds. Let $M$ be an irreducible $U^{[r]}_\chi(G)$-module and $N$ an irreducible $\Di(G_{(r)})$-module such that $M\cong N\otimes\Hom_{\tiny \Di(G_{(r)})}(N,M)$ as $\Di(G_{(r)})$-modules. Then $p^{\dim(G\cdot\chi)/2}$ divides $\dim \Hom_{\tiny \Di(G_{(r)})}(N,M)$.
\end{prop}

%

%
%
%

Whether one approaches this question through Theorem~\ref{Decomp} and Premet's theorem, or through Theorem~\ref{Decomp3}, Lemma~\ref{End} and Premet's theorem, all that remains is to show that for $x\in\fg$, $x^p-x^{[p]}$ acts on $\Hom_D(N,M)$ as $\chi(x)^{p}$. 

For the latter technique, given $\delta\in \Di_{p^r}^{+}(G)$, we know that $\delta^{\otimes p}-\delta^p$ is central in $\Upr$. Hence, the map $\rho(\delta):\Upr\otimes_{D}N\to \Upr\otimes_{D}N$ given by left multiplication by $\delta^{\otimes p}-\delta^p$ is a $\Upr$-module endomorphism of $\Upr\otimes_{D}N$, and so lies inside $E$. However, as we know that $M$ is a $U^{[r]}_\chi(G)$-module, $\rho(\delta)\in E$ acts on $\Hom_{\tiny D}(N,M)$ as multiplication by $\chi(\delta)^p$.

Hence, to show that $\Hom_{\tiny D}(N,M)$ is a $U_\chi(\fg)$-module, we just need that, for $\alpha\in\Phi$, $\ve_{\alpha}^p$ maps to $\rho((\ve_{\alpha}^{(p^r)})^{\otimes p})$ and, for $1\leq t\leq d$, $\vh_{t}^p-\vh_{t}$ maps to $\rho( \binom{\vh_{t}}{p^r}^{\otimes p}-{\binom{\vh_{t}}{p^r}})$ under the isomorphism $U(\fg)\cong E$. 

As a result, completing the proof of this proposition relies on a more detailed understanding of the isomorphism $U(\fg)\cong E$. In fact, this is precisely the same question as understanding how the elements $\ve_{\alpha}^p\in U(\fg)$, for $\alpha\in\Phi$, and $\vh_{t}^p-\vh_{t}$, for $1\leq t\leq d$, act on $\Hom_{\tiny \Di(G_{(r)})}(N,M)$ under the original action. This shall be explored in the sequel \cite{West}.

\end{document}